\documentclass[a4paper,reqno,final]{amsart}

\usepackage[utf8]{inputenc}
\usepackage[T1]{fontenc}
\usepackage{amsmath, amssymb, amsthm}
\usepackage{dsfont}
\usepackage{mathtools}
\usepackage[english]{babel}
\usepackage{enumitem}

\usepackage{stmaryrd}
\usepackage{tikz}
\usepackage{tikz-cd}

\usepackage[unicode,bookmarksopen]{hyperref}

\numberwithin{equation}{section}

\newcommand{\bbC}{\mathbb{C}}
\newcommand{\bbN}{\mathbb{N}}

\newcommand{\bbR}{\mathbb{R}}

\newcommand{\bbZ}{\mathbb{Z}}
\newcommand{\calA}{\mathcal{A}}
\newcommand{\calF}{\mathcal{F}}
\newcommand{\calL}{\mathcal{L}}

\newcommand{\calS}{\mathcal{S}}
\newcommand{\calT}{\mathcal{T}}
\newcommand{\calU}{\mathcal{U}}
\newcommand{\calX}{\mathcal{X}}

\theoremstyle{definition}
\newtheorem{definition}{Definition}[section]
\newtheorem{remark}[definition]{Remark}

\newtheorem{remarks}[definition]{Remarks}
\newtheorem{example}[definition]{Example}

\newtheorem{construction}[definition]{Construction}
\newtheorem{assumptions}[definition]{Assumptions}
\newtheorem{open_problem}[definition]{Open Problem}

\theoremstyle{plain}
\newtheorem{proposition}[definition]{Proposition}
\newtheorem{lemma}[definition]{Lemma}
\newtheorem{theorem}[definition]{Theorem}
\newtheorem{corollary}[definition]{Corollary}

\DeclareMathOperator{\linspan}{span}

\DeclareMathOperator{\Id}{Id}
\DeclareMathOperator{\conv}{conv}
\DeclareMathOperator*{\stoplim}{stop-lim}
\DeclareMathOperator*{\wlim}{w-lim}

\newcommand{\IZ}{\mathbb{Z}}                                

\newcommand{\IN}{\mathbb{N}}

\newcommand{\IE}{\mathbb{E}}

\newcommand{\abs}[1]{\left\lvert#1\right\rvert}

\newcommand{\norm}[1]{\left\lVert#1\right\rVert}
\newcommand{\normalnorm}[1]{\lVert#1\rVert}

\newcommand{\R}[2][\empty]{
	\ifthenelse{\equal{#1}{\empty}}
		{\mathcal{R}\left\{#2\right\}}
		{\mathcal{R}_{#1}\left\{#2\right\}}
}

\let\temp\epsilon
\let\epsilon\varepsilon
\let\varepsilon\temp

\allowdisplaybreaks

\title{A Toolkit for Constructing Dilations on Banach Spaces}
\author{Stephan Fackler}
\author{Jochen Glück}
\address{Institute of Applied Analysis, Ulm University, Helmholtzstr.\ 18, 89069 Ulm}
\email{stephan.fackler@alumni.uni-ulm.de, jochen.glueck@uni-ulm.de}
\date{\today}
\thanks{The first author was supported by the DFG grant AR 134/4-1 ``Regularität evolutionärer Probleme mittels Harmonischer Analyse und Operatortheorie''. Both authors were supported by a scholarship within the scope of the Landesgraduiertenf\"orderung Baden--W\"urttemberg, Germany, while some of the work on this article was done.}

\begin{document}

\begin{abstract}
	We present a completely new structure theoretic approach to the dilation theory of linear operators. Our main result is the following theorem: if $X$ is a super-reflexive Banach space and $T$ is contained in the weakly closed convex hull of all invertible isometries on $X$, then $T$ admits a dilation to an invertible isometry on a Banach space $Y$ with the same regularity as $X$. The classical dilation theorems of Sz.-Nagy and Akcoglu-Sucheston are easy consequences of our general theory.
\end{abstract}

\maketitle

\section{Introduction} \label{sec:introduction}

Consider a bounded linear operator $T$ on a Banach space $X$. We say that $T$ has a dilation to a Banach space $Y$ if there exist linear contractions $J\colon X \to Y$ and $Q\colon Y \to X$ and an invertible linear isometry $U\colon Y \to Y$ such that
\begin{equation}
	T^n = QU^nJ \qquad \text{ for all } n\in \bbN_0. \label{eq:dilation}
\end{equation} 
It follows from this definition that $T$ is necessarily contractive. Constructing a dilation for a given operator $T$ is a very subtle endeavor since the questions whether such a dilation exists and, granted that it exists, whether it is useful in applications depend crucially on the choice of $Y$. 

In concrete applications one is usually interested in $Y$ having the same regularity as the original space $X$. If, for instance, $X$ is a Hilbert space or an $L^p$-space, then $Y$ should be out of the same class. One therefore requires $Y$ to be out of a prescribed class of Banach spaces $\calX$. The question whether a dilation to such a space exists is a delicate one, and for some operators it may not be possible to find a dilation in the class $\calX$. For instance, by a result of Sz.-Nagy every Hilbert space contraction has a dilation in the class of Hilbert spaces~\cite[Theorem~1.1]{Pis01}, whereas for $p \in (1, \infty) \setminus \{2\}$, as a direct consequence of~\cite[Theorem~3]{Pel81}, there exist contractions on $L^p$ that do not have a dilation in the class of all $L^p$-spaces. 

The applications of dilation results are plenty and profound. For example, Sz.-Nagy's dilation, in conjunction with the spectral theorem for normal operators, implies von Neumann's inequality: every Hilbert space contraction $T \in \mathcal{L}(H)$ satisfies
\begin{equation*}
	\norm{p(T)}_{\mathcal{L}(H)} \le \sup_{\abs{z} = 1} \norm{p(z)}
\end{equation*}
for all complex polynomials $p$. A second celebrated dilation theorem due to Akcoglu and Sucheston~\cite{AkcSuc77} states that every positive contraction on a reflexive $L^p$-spaces has a dilation to a positive invertible isometry on another $L^p$-space for the same $p$. As a consequence one obtains for positive contractions on $L^p$ the almost everywhere convergence of ergodic means~\cite{Akc75} and the validity of Matsaev's inequality~\cite[Theorem~9]{Pel81}. Further, Akcoglu and Sucheston's theorem is used to deduce fundamental properties of both discrete and continuous positive contractive semigroups on reflexive $L^p$-spaces such as the boundedness of their $H^{\infty}$-calculus (\cite[Remark~4.9c)]{Wei01} and~\cite[Theorem~8.3]{Mer14}) or maximal $L^q$-regularity~\cite[Theorem~1.3]{Blu01b}.

\subsection*{A structure theoretic approach}

Usually, proofs of dilation theorems are explicitly adapted to the spaces and operators under consideration and make massive use of their particular properties. As a concrete instance of this observation, we refer to the detailed presentation of Akcoglu and Sucheston's dilation theorem given in~\cite[Chapters~2~\&~3]{Fendler1998}. 
This makes the proof difficult to understand at an abstract level and, moreover, the proof does not really clarify the exact role of the geometry of the underlying Banach space $X$. 
Furthermore, although at least three different proofs for Akcoglu and Sucheston's theorem are nowadays known and have been known for several decades, namely those in \cite{AkcSuc77}, \cite{Pel81} and~\cite{NagPal82}, no general technique seems to be available to construct dilations on more general spaces. 
To the best of our knowledge, no dilation result is known for any non-trivial class of contractions on general super-reflexive or UMD Banach spaces.

In order to address those issues we develop a new, structure theoretical approach to dilation problems. We introduce the concept of simultaneous dilations for a set of operators which turns out to be the right framework for proving general dilation theorems. Based on this concept, we develop a toolkit for constructing dilations on general classes of reflexive Banach spaces. This toolkit essentially consists of topological and algebraical closedness results. Given a set of simultaneously dilating operators, our main result asserts that its convex hull also admits a simultaneous dilation to a Banach space with the same regularity. Moreover, it is easy to see that the existence of dilations is preserved by strong operator limits. As a consequence  the weakly closed convex hull of all invertible isometries on a Banach space simultaneously dilates to a space with similar regularity (Corollary~\ref{cor:convex-combinations-of-isometries}). Hence, proving dilation theorems reduces to the task of finding approximation theorems on Banach spaces, and it is only here where the particular geometry of the Banach space comes into play. As a consequence, we achieve the following three goals:

\begin{enumerate}
	\item[(i)] We show that every convex combination of invertible isometries on a super-reflexive or UMD space dilates to an invertible isometry on another space of the same class (Theorem~\ref{thm:super-reflexive-spaces} and Corollary~\ref{cor:umd-spaces}).
	\item[(ii)] We give a new proof of Akcoglu and Sucheston's dilation theorem (Subsection~\ref{subsec:L-p}).
	\item[(iii)] We generalize their theorem by establishing \emph{simultaneous} dilations (Theorem~\ref{thm:akcoglu-simultaneously}).
\end{enumerate}

\subsection*{Dilation and regularity}

As pointed out above, the existence of a dilation crucially depends on the considered class of Banach spaces. Let us demonstrate by a simple construction that, for \emph{every} contraction $T$ on a Banach space $X$, there exists a Banach space $Y$ such that $T$ dilates to an invertible isometry on $Y$.

\begin{construction} \label{constr:simple-dilation-ell-1}
	Let $T\colon X \to X$ be a contractive linear operator on a Banach space $X$. Choose $Y \coloneqq \ell^{1}(\bbZ;X)$, let $U\colon \ell^1(\bbZ;X) \to \ell^1(\bbZ;X)$ be the right shift $(x_n)_{n \in \bbZ} \mapsto (x_{n-1})_{n \in \bbZ}$, let $J\colon X \to Y$ be the injection into the zeroth component, meaning that $Jx = (\ldots, 0, x, 0 , \ldots)$ for all $x \in X$, and define $Q\colon Y \to X$ by
\begin{align*}
	Q(x_n)_{n \in \bbZ} = \sum_{n = 0}^\infty T^n x_n
\end{align*}
for all $(x_n)_{n \in \bbZ} \in Y$. Then $U$ is an invertible isometry on $Y$, the operators $J$ and $Q$ are contractive and~\eqref{eq:dilation} holds.
\end{construction}

A similar -- and somewhat dual -- construction can be found on $Y = \ell^{\infty}(\IZ;X)$; see also \cite{Stroescu1973} for a related construction. The above construction demonstrates why dilation theory is all about the choice of the space $Y$. It is easy to construct a dilation to \emph{some} Banach space $Y$; this space, though, does not inherent any regularity properties such as reflexivity from $X$. 

In this context it is also worthwhile pointing out that, in case that $X$ is an $L^1$-space, the space $Y$ from Construction~\ref{constr:simple-dilation-ell-1} is an $L^1$-space, too. However, if the underlying measure space of $X$ is finite, say a probability space, the underlying measure space of $Y$ is only $\sigma$-finite, though. This drawback prevents us, for instance, from applying Construction~\ref{constr:simple-dilation-ell-1} in probability theory.

\subsection*{Outline of the paper} 

In Section~\ref{sec:tool-kit} we introduce our framework and then state our main result, Theorem~\ref{thm:main-result}, and an important consequence, Corollary~\ref{cor:convex-combinations-of-isometries}. 
In Section~\ref{sec:elementary-properties} we show some simple stability properties of dilations. In Sections~\ref{sec:convex-combinations} and~\ref{sec:convex-combinations-simultaneous} we prove that dilations are well-behaved with respect to convex combinations. We consciously chose to introduce a bit of redundancy in the Sections~\ref{sec:convex-combinations} and~\ref{sec:convex-combinations-simultaneous} to make the quite technical proof of Theorem~\ref{thm:convex-combinations-simultaneous} more transparent. In Section~\ref{sec:super-properties} we use our main result to deduce dilation results on Banach spaces satisfying certain regularity properties such as super-reflexive spaces and UMD-spaces. In Section~\ref{sec:L-p} we discuss how our main result can be used to reprove the Akcoglu--Sucheston dilation theorem on $L^p$-spaces and to even obtain a generalization of it; we also show that the Sz.-Nagy dilation theorem on Hilbert spaces is a direct consequence of our approach. We conclude with an outlook in Section~\ref{sec:outlook}. Appendix~\ref{appendix:group-theory} contains a few simple results from group theory used in the proofs of Theorems~\ref{thm:convex-combination} and~\ref{thm:convex-combinations-simultaneous}.

\subsection*{Related literature}

For information on the Sz.-Nagy's dilation theorem on Hilbert spaces we refer the reader to \cite{SFBK10}. For several predecessors of Akcoglu and Sucheston's result we refer the reader to the references in~\cite{AkcSuc77}. For a certain class, the Ritt operators, the existence of dilations can be characterized (see~\cite[Theorem~4.1]{ArhFacMer17} and~\cite[Theorem~4.8]{ArhMer14}). A dilation theorem on rearrangement invariant Banach function spaces due to Peller can be found in \cite[Section~6, Theorem~7]{Pel81}. Dilations on $L^1$-spaces can, for example, be constructed by using methods from the theory of Markov processes; see for instance \cite{Kern1977}; moreover, the dilations constructed in this way are well-behaved with respect to spectral properties \cite[Section~5]{Kern1977}. Dilation results on non-commutative spaces such as $W^*$-algebras are of importance in quantum physics; a systematic treatment of them was initiated by K\"ummerer in the 1980s; see for example~\cite{Kuemmerer1985a}. The situation on non-commutative $L^p$-spaces is for instance discussed in~\cite{Junge2007} and in the recent paper \cite{ArhancetPreprint}.

\subsection*{Preliminaries}

Throughout the paper the underlying scalar field of all Banach spaces is allowed to be either $\bbR$ or $\bbC$, but we assume that it is the same field for all occurring Banach spaces. If $X$ is a Banach space, then we denote the space of all bounded linear operators from $X$ to $X$ by $\calL(X)$. For the product of finitely many operators $T_1,\dots,T_n \in \calL(X)$ we use the notation
\begin{align*}
	\prod_{k=1}^n T_k \coloneqq T_1 \dots T_n,
\end{align*}
i.e.\ the operator with the lowest index is placed left in the product. Moreover, we use the common convention that $\prod_{k=1}^0 \coloneqq \Id_X$, i.e.\ the empty product is defined to be the identity operator.

\section{The dilation toolkit and main results} \label{sec:tool-kit}

In this section we introduce the framework for our theory, we present our main results in Theorem~\ref{thm:main-result} and Corollary~\ref{cor:convex-combinations-of-isometries} and we give some fundamental characterizations of dilation properties.

\subsection*{The framework}

Let $I$ be a non-empty index set, let $(X_i)_{i \in I}$, $(Y_i)_{i\in I}$ be two families of Banach spaces and let $(T_i)_{i \in I}$ be a family with $T_i \in \mathcal{L}(X_i,Y_i)$ for all $i \in I$. For every ultrafilter $\calU$ on $I$ we denote the ultraproducts of the families $(X_i)_{i \in I}$ and $(Y_i)_{i \in I}$ along $\calU$ by $\prod_{\calU} X_i$ and $\prod_{\calU} Y_i$, respectively; the ultraproduct of the operator family $(T_i)_{i \in I}$ along $\calU$ is denoted by $\prod_{\calU} T_i$. Let $x_i \in X_i$ for each index $i \in I$. Then we denote the equivalence class of the family $(x_i)_{i \in I}$ within $\prod_{\calU} X_i$ by $(x_i)_\calU$. If $X_i = X$ for a Banach space $X$ and all $i \in I$, then we use the abbreviation $X^\calU \coloneqq \prod_\calU X_i$; in this case the ultraproduct $X^\calU$ is called a ultrapower of $X$. In a canonical way every operator $T \in \calL(X)$ induces an operator on $X^\calU$ which we denote by $T^\calU$. For an introduction to the theory of ultraproducts we refer the reader to~\cite{Hei80} and~\cite[Section~8]{DJT95}.

Let $p \in (1,\infty)$ and $n \in \bbN$. For a Banach space $X$ we denote the space $X^n$, endowed with the norm
\begin{align*}
	\|(x_1,\dots,x_n)\|_p \coloneqq \big(\sum_{k=1}^n \|x_k\|^p\big)^{1/p},
\end{align*}
by $\ell^p_n(X)$. Throughout we often consider classes of Banach spaces that fulfill a set of conditions adopted to our constructions. Fix a number $p \in (1,\infty)$.

\begin{assumptions} \label{ass:framework}
	We say that a class of Banach spaces $\calX$ fulfills Assumptions~\ref{ass:framework} if the following conditions hold.
	\begin{enumerate}
		\item[(a)] The class $\calX$ is stable with respect to finite $\ell^p$-powers, i.e.\ for every $X \in \calX$ and every $n \in \bbN$ we have $\ell^p_n(X) \in \calX$.
		\item[(b)] The class $\calX$ is ultra-stable, i.e.\ for every family of spaces $(X_i)_{i \in I}$ in $\calX$ and every ultrafilter $\calU$ on $I$ we have $\prod_{\calU} X_i \in \calX$.
		\item[(c)] Every space $X \in \calX$ is reflexive.
	\end{enumerate}
\end{assumptions}

Let us mention two simple but important classes which fulfills the Assumptions~\ref{ass:framework}.

\begin{example}\label{ex:class_hilbert_space}
	Let $p = 2$. Then the class of all Hilbert spaces fulfills Assumptions~\ref{ass:framework}.
\end{example}

\begin{example}
	Fix $p \in (1,\infty)$. Then the class of all $L^p$-spaces (over arbitrary measure spaces) fulfills Assumptions~\ref{ass:framework}.
\end{example}

\subsection*{Super-reflexive spaces}

Recall that a Banach space $Z$ is called \emph{super-reflexive} if every ultrapower of $Z$ is reflexive. If a class of Banach spaces $\calX$ fulfills Assumptions~\ref{ass:framework} then, as a consequence of~(b) and~(c), all spaces in $\calX$ are super-reflexive. Conversely, we now show that, for every super-reflexive Banach space $Z$, there exists a class $\mathcal{X}_Z$ of Banach spaces that contains $Z$ and fulfills Assumptions~\ref{ass:framework}. We need the following terminology.

\begin{definition}\label{def:finitely_representable}
	A Banach space $X$ is \emph{finitely representable} in a second Banach space $Z$, and we write $X \xhookrightarrow{f} Z$ for this, if for every finite dimensional subspace $E \subseteq X$ and every $\epsilon > 0$ there exists a subspace $F \subseteq Z$ and an isomorphism $u\colon E \to F$ with $\norm{u} \normalnorm{u^{-1}} \le 1 + \epsilon$.
\end{definition}

The following elementary observation is quite useful for our purposes.

\begin{lemma}\label{lem:ell_2_finitely_repres}
	Let $X$ and $Z$ be Banach spaces with $X \xhookrightarrow{f} Z$. Then one has $\ell^2_n(X) \xhookrightarrow{f} \ell^2(Z)$ for all $n \in \IN$.
\end{lemma}
\begin{proof}
	Let $n \in \IN$, $\epsilon > 0$ and let $E \subseteq \ell^2_n(X)$ be a finite dimensional subspace. Choose a basis $y_1, \ldots, y_m$ of $E$. We write $y_k = (x_{k1}, \ldots, x_{kn})$ for all $k = 1, \ldots, m$ and we set 
	\begin{equation*}
		\tilde{E} = \linspan \big\{ x_{kl}: k \in \{ 1, \ldots, m \}, \, l \in \{ 1, \ldots, n \} \big\} \subseteq X.
	\end{equation*}
	By assumption, there exist a subspace $\tilde{F} \subseteq Z$ and an isomorphism $\tilde{u}\colon \tilde{E} \to \tilde{F}$ with $\norm{\tilde{u}} \normalnorm{\tilde{u}^{-1}} \le 1 + \epsilon$. Let $z_k = (\tilde{u}(x_{k1}), \ldots, \tilde{u}(x_{kn}),0, \ldots) \in \ell^2(Z)$ for $k = 1, \ldots, m$ and $F = \linspan \{ z_1, \ldots, z_m \} \subseteq \ell^2(Z)$. Then $u\colon \sum_{k=1}^m a_k y_k \mapsto \sum_{k=1}^m a_k z_k$ is an isomorphism from $E$ onto $F$ with $\norm{u} \normalnorm{u^{-1}} \le 1 + \epsilon$.
\end{proof}

It is well-known that a Banach space $Z$ is super-reflexive if and only if every Banach space $X$ that is finitely representable in $Z$ is reflexive. We can now show that every super-reflexive Banach space is contained in a class of spaces which fulfill the Assumptions~\ref{ass:framework}.

\begin{proposition}\label{prop:X_Z_satisfies_assumptions}
	Let $Z$ be a super-reflexive Banach space and define the class
	\begin{equation*}
		\calX_Z \coloneqq \left\{ X \text{ Banach space}: X \xhookrightarrow{f} \ell^2(Z) \right\}.
	\end{equation*}
	Then $\calX_{Z}$ contains $Z$ and fulfills Assumptions~\ref{ass:framework} for $p=2$.
\end{proposition}
\begin{proof}
	Since $Z$ is isometrically isomorphic to a subspace of $\ell^2(Z)$ we clearly have $Z \in \calX_Z$. Let us now show that $\calX_Z$ fulfills Assumptions~\ref{ass:framework}.
	
	Since $Z$ is super-reflexive, so is $\ell^2(Z)$. Hence, every space $X \in \calX_Z$ is reflexive, which shows~(c). For (a) let $n \in \IN$ and $X \in \calX_Z$. By Lemma~\ref{lem:ell_2_finitely_repres} one has $\ell^2_n(X) \xhookrightarrow{f} \ell^2(\ell^2(Z)) = \ell^2(Z)$, where the last equality holds in the sense of isometric isomorphisms. The remaining property~(b) of Assumptions~\ref{ass:framework} follows from the fact that every ultra power of a Banach space is finitely representable in the space itself \cite[Lemma~11.66]{Pis16}.
\end{proof}

In Section~\ref{sec:super-properties} we combine Proposition~\ref{prop:X_Z_satisfies_assumptions} with so-called \emph{super-properties} of Banach spaces to find classes of Banach spaces which fulfill Assumptions~\ref{ass:framework} and which have, at the same time, further regularity properties.

\subsection*{Simultaneously dilating operators} Let us now define sets of \emph{simultaneously dilating} operators. 

\begin{definition} \label{def:dilations}
	Let $\calX$ be a class of Banach spaces and let $X \in \calX$.
	\begin{enumerate}
		\item[(a)] An operator $T \in \calL(X)$ has a \emph{dilation} in $\calX$ if there exist a space $Y \in \calX$, linear contractions $J\colon X \to Y$, $Q\colon Y \to X$ and a linear invertible isometry $U \in \calL(Y)$ such that for all $n \in \bbN_0$
		\begin{align}
			\label{eq:single-dilation}
			T^n = QU^nJ.
		\end{align}
		\item[(b)] A set of operators $\calT \subseteq \calL(X)$ has a \emph{simultaneous dilation} in $\calX$ if there exist a space $Y \in \calX$, linear contractions $J\colon X \to Y$, $Q\colon Y \to X$ and invertible isometries $U_T \in \calL(Y)$ (for $T \in \calT$) such that the equality
		\begin{align}
			\label{eq:simultaneous-dilation}
			T_1 \cdots T_n = Q \, U_{T_1} \cdots U_{T_n} \, J
		\end{align}
		holds for all $n \in \bbN_0$ and all $T_1,\dots, T_n \in \calT$.
	\end{enumerate}
\end{definition}

For a proper reading of part~(b) of the above definition it is important to recall that we agreed on the empty product to be the identity operator.

If $T$ has a dilation, then it follows from the equality $T^n = QU^nJ$ for $n=0$ that $QJ = \Id_X$. Hence, $J$ is automatically an isometry and $JQ$ is a contractive projection on $Y$ with range $J(X)$. Thus, $X$ may be considered as a subspace of $Y$ (via $J$) and $Q$ may be considered as a projection from $Y$ onto this subspace. Moreover, $T = QUJ$ implies that $T$ is contractive, i.e.\ only contractive operators can have a dilation in our sense. Similar observations hold for simultaneous dilations.

We point out that our notion of a dilation differs slightly from the definition which is, for instance, used by Akcoglu and Sucheston in~\cite{AkcSuc77}. Yet, it is easy to see that if an operator has a dilation in the sense of Definition~\ref{def:dilations}(a), then it also has a dilation in the sense of \cite{AkcSuc77}.

If $X$ is a Banach space taken from a given class $\calX$, then the set of all invertible isometries always has a simultaneous dilation in $\calX$. 

\begin{example} \label{ex:ismoetries-have-simultaneous-dilation}
	Let $\calX$ be a class of Banach spaces, let $X \in \calX$ and $\calT \subseteq \calL(X)$ be a set of linear invertible isometries. Then $\calT$ has a simultaneuous dilation in $\calX$. Indeed, simply take $Y = X$, $J = Q = \Id$ and $U_T = T$ for all $T \in \calT$.
\end{example}

The whole point of our approach is to first find a (not too small) set of operators which admits a simultaneous dilation and then to construct, out of this set, a larger set that has a dilation by proving stability results. Example~\ref{ex:ismoetries-have-simultaneous-dilation} shows that we can always start with the set of invertible isometries; proving stability results is much more involved. Our main result, Theorem~\ref{thm:main-result}, is of this type.

\subsection*{The main result}

Theorem~\ref{thm:main-result} below is our main result; in conjunction with Example~\ref{ex:ismoetries-have-simultaneous-dilation} it yields simultaneous dilations for large sets of operators.

\begin{theorem} \label{thm:main-result}
	Fix $p \in (1,\infty)$ and let $\calX$ be a class of Banach spaces which fulfills Assumptions~\ref{ass:framework}. Let $X \in \calX$ and let $\calT \subseteq \calL(X)$ be a set of operators that has a simultaneous dilation in $\calX$. Then the weak operator closure of the convex hull $\conv(\calT)$ also has a simultaneous dilation in $\calX$.
\end{theorem}

As a consequence of Example~\ref{ex:ismoetries-have-simultaneous-dilation} and Theorem~\ref{thm:main-result} we immediately obtain the following corollary.

\begin{corollary} \label{cor:convex-combinations-of-isometries}
	Fix $p \in (1,\infty)$ and let $\calX$ be a class of Banach spaces which fulfils Assumptions~\ref{ass:framework}. Let $X \in \calX$ and let $\calT \subseteq \calL(X)$ denote the weak operator closure of the convex hull of all linear invertible isometries on $X$. Then $\calT$ has a simultaneous dilation in $\calX$. In particular, every $T \in \calT$ has a dilation in $\calX$.
\end{corollary}

The above corollary demonstrates how our approach uncouples dilation theory from any geometric considerations. The dilation result in Corollary~\ref{cor:convex-combinations-of-isometries} holds for all classes of Banach spaces that fulfill the rather mild Assumptions~\ref{ass:framework}, with no regard of the special choice of the spaces in $\calX$. However, in order to apply Corollary~\ref{cor:convex-combinations-of-isometries} to concrete operators one has to determine the weak operator closure of the convex hull of all invertible isometries. This is a pure approximation theoretical task and it is here where the special geometry of the Banach space $X$ comes into play.

\begin{remarks} \label{rem:positive-morphisms}
	(a) Sometimes, it is desirable to restrict not only the class of Banach spaces $\calX$ out of which the space $Y$ in Definition~\ref{def:dilations} is taken, but also the choice for the operators $J$, $Q$ and $U_T$. A typical situation of this type is a follows:
	
	Assume that all spaces in $\calX$ are Banach lattices and that $\calT$ consists of positive operators only. Then we would like to construct a \emph{positive} dilation for an operator $T \in \calT$ or, in greater generality, a \emph{positive} simultaneous dilation of $\calT$. By this we mean that the operators $J$ and $Q$ from Definition~\ref{def:dilations} should not only be contractive, but also positive, and that the operators $U_T$ (for $T \in \calT$) should not only be invertible isometries, but also lattice isomorphisms. 
	Under the assumption that all spaces in $\calX$ are Banach lattices and that all operators in $\calT$ are positive, all our results yield \emph{positive} (simultaneous) dilations instead of only dilations -- provided that the assumptions are adjusted in the obvious way. For instance, in Theorem~\ref{thm:main-result} one has to assume that $\calT$ has a simultaneous positive dilation in $\calX$ and in Corollary~\ref{cor:convex-combinations-of-isometries} one has to assume that $\calT$ is the weak operator closure of the convex hull of all positive invertible isometries.
	
	(b) In view of remark~(a) notice that an invertible isometry $T \in \calL(X)$ on a Banach lattice $X$ is automatically a lattice isomorphism if it is positive; this is a theorem of Abramovich, see for instance \cite{Abramovich1988} or \cite[Theorem~2.2.16]{Emelyanov2007}.
	
	(c) One could also formalize the idea discussed in~(a) by using the language of category theory. One then considers a class of Banach spaces $\calX$ which fulfills Assumptions~\ref{ass:framework} and a class $\mathcal{M}$ of morphisms that satisfies certain stability conditions. In Definition~\ref{def:dilations}, we would require the operators $T \in \calT$, as well as $J$, $Q$ and $U_T$, to be in $\mathcal{M}$.
	
	If we chose $\mathcal{M}$ to be the class of all linear contractions, we would obtain the dilations introduced in Definition~\ref{def:dilations}; if we chose $\calX$ to be a class of Banach lattices and $\mathcal{M}$ to be the class of positive contractions, we would obtain the notion of a \emph{positive dilation} as discussed in~(a).
	
	One could even go further and consider two classes of morphisms $\mathcal{M}_1$ and $\mathcal{M}_2$, where the operators $J$ and $Q$ are required to be contained in $\mathcal{M}_2$ while the operators in $\calT$ and the operators $U_T$ are in $\mathcal{M}_1$. If we chose $\calX$ to be a class of Banach lattices, $\mathcal{M}_2$ to be the class of all positive contractions and $\mathcal{M}_1$ to be the class of all regular operators with regular norm at most $1$, we would thus obtain a more general class of dilations. These are of relevance in a version of the Akcoglu--Sucheston theorem on $L^p$-spaces; see for
	instance \cite[p.~58ff.]{CoiRocWei78} or \cite[Section~3]{Pel81}.
	
	We shall however not pursue this category theoretical approach here since we wish to state all our results in a concrete and easily accessible way. The interested reader won't find it difficult to restate our results in the language of category theory.
\end{remarks}

Next we note that one can replace the convex hull in Theorem~\ref{thm:main-result} by somewhat larger sets.

\begin{remarks} \label{rem:subconvex-and-absolutely-convex-hull}
	(a) Let $V$ be a vector space over the field $\bbR$ or $\bbC$. The \emph{absolutely convex hull} of a subset $C \subseteq V$ is given by
	\begin{align*}
		\bigg\{\sum_{k=1}^n \lambda_k v_k: \; n \in \bbN, \; v_1,\dots,v_n \in C, \; \lambda_1,\dots,\lambda_n \in \mathbb{K} \text{ and } \sum_{k=1}^n \lvert \lambda_k \rvert \le 1\bigg\}.
	\end{align*}
	It coincides with the convex hull of the set $\{\lambda v: \; v \in C, \; \lambda \in \mathbb{K} \text{ and } \lvert \lambda \rvert = 1 \}$. On the other hand it is easy to see that, under the assumptions of Theorem~\ref{thm:main-result}, the set
	\begin{align*}
		\big\{ \lambda T: \; T \in \calT, \; \lambda \in \mathbb{K} \text{ and } \lvert \lambda \rvert = 1 \big\}
	\end{align*}
	has a simultaneous dilation in $\calX$; here $\mathbb{K}$ is the underlying scalar field of the spaces in $\calX$. Hence, Theorem~\ref{thm:main-result} implies that the weak operator closure of the absolutely convex hull of $\calT$ has a simultaneous dilation in $\calX$.
	
	(b) The arguments given in~(a) do no longer apply if one is looking for positive dilations as discussed in Remark~\ref{rem:positive-morphisms}(a). Indeed, if an operator $T$ has a positive dilation in some Banach lattice, then its negative $-T$ does not have a positive dilation, though. However, one can still combine Theorem~\ref{thm:main-result} with Proposition~\ref{prop:zero-operator} below. This way one obtains that the weak operator closure of the \emph{subconvex hull}
	\begin{align*}
		\bigg\{ \sum_{k=1}^n \lambda_k T_k: \; n \in \bbN,\; T_1, \dots, T_n \in \calT, \; \lambda_1,\dots,\lambda_n \in [0,1] \text{ and } \sum_{k=1}^n \lambda_k \le 1 \bigg\}
	\end{align*}
	admits a positive simultaneous dilation in $\calX$ if $\calT$ does so. This works since the proof of Proposition~\ref{prop:zero-operator} does not destroy the positivity structure of the dilation. In particular, one obtains the following version of Corollary~\ref{cor:convex-combinations-of-isometries}:
	
	Let $X$ be a class of Banach lattices that fulfills Assumptions~\ref{ass:framework}. Let $X \in \calX$ and let $\calT \subseteq \calL(X)$ be the set of all positive, invertible isometries. Then the weak operator closure of the subconvex hull of $\calT$ has a positive simultaneous dilation in $\calX$.
\end{remarks}

The rest of Section~\ref{sec:tool-kit}, as well as Sections~\ref{sec:elementary-properties}, \ref{sec:convex-combinations} and~\ref{sec:convex-combinations-simultaneous}, are devoted to the proof of Theorem~\ref{thm:main-result}. We start with two useful characterizations of simultaneous dilations in the remaining part of Section~\ref{sec:tool-kit}: a ``finitary characterization'' for simultaneous dilations which shows that it is actually sufficient to consider finite sets of operators and a result which shows that it suffices to establish the dilation equality only for powers/monomials of bounded degree.

In Section~\ref{sec:elementary-properties} we proceed with our preparations for the proof of Theorem~\ref{thm:main-result}. Note it suffices to show that simultaneous dilations behave well with respect to strong operator closures and with respect to taking convex combinations since the weak and the strong operator closure of a convex set coincide \cite[Corollary VI.1.5]{DunSch1958}. The stability result for strong closures is rather simple and will be given in Proposition~\ref{prop:stop-limits}. In Section~\ref{sec:elementary-properties} we also discuss a few further elementary properties of simultaneous dilations such as stability with respect to operator multiplication. The stability result for convex combinations is, however, much more involved. Therefore, we first prove in Section~\ref{sec:convex-combinations} that, if a set $\calT$ of operators has a simultaneous dilation in a class $\calX$, then every single convex combination of operators from $\calT$ also has a dilation in $\calX$. In Section~\ref{sec:convex-combinations-simultaneous} we then prove that the convex hull of $\calT$ actually has a simultaneous dilation. All results of Section~\ref{sec:convex-combinations} can be seen as special cases of results from Section~\ref{sec:convex-combinations-simultaneous}, so we introduce a bit of redundancy by considering those two cases separately. However, given the rather technical computations in those sections, we think that the reader might benefit from this redundancy.

\subsection*{A finitary characterization of dilations}

The following proposition shows that, when dealing with simultaneous dilations, one can restrict to finite sets of operators.

\begin{proposition} \label{prop:finitary}
	Fix $p \in (1,\infty)$ and let $\calX$ be a class of Banach spaces that fulfills Assumptions~\ref{ass:framework}. Let $X \in \calX$ and let $\calT \subseteq \calL(X)$. Then the following are equivalent:
	\begin{enumerate}
		\item[(i)] $\calT$ has a simultaneous dilation in $\calX$.
		\item[(ii)] Every finite subset of $\calT$ has a simultaneous dilation in $\calX$.
	\end{enumerate}
\end{proposition}

To keep the proof of Proposition~\ref{prop:finitary} as transparent as possible we first show the following lemma which is, in a way, a more abstract version of the proposition.

\begin{lemma} \label{lem:extend-dilations-along-a-net}
	Fix $p \in (1,\infty)$, let $\calX$ be a class of Banach spaces that fulfills Assumptions~\ref{ass:framework} and let $X \in \calX$. Consider a net $(\calT_i)_{i \in I}$ of subsets of $\calL(X)$ and assume that this net is monotone, i.e.\ $\calT_j \supseteq \calT_i$ whenever $j \ge i$. If each set $\calT_i$ has a simultaneous dilation in $\calX$, then $\bigcup_{i \in I} \calT_i$ has a simultaneous dilation in $\calX$.
\end{lemma}
\begin{proof}
	Choose an ultrafilter $\calU$ on $I$ that contains the filter base $\big\{ \{j \in I: \, j \ge i\} : \; i \in I \big\}$.
	For each $i \in I$ we can find a Banach space $Y_i \in \calX$, contractions $J_i\colon X \to Y_i$ and $Q_i\colon Y_i \to X$ and invertible isometries $U_{i,T} \in \calL(Y_i)$ (for $T \in \calT_i$) with
	\begin{align*}
		T_1 \cdots T_n = Q_i \, U_{i,T_1} \cdots U_{i,T_n} J_i
	\end{align*}
	for all $n \in \bbN_0$ and all $T_1,\dots,T_n \in \calT_i$. Now, define $Y \coloneqq \prod_\calU Y_i$ as well as $J \coloneqq \prod_\calU J_i \colon X^\calU \to Y$ and $Q \coloneqq \prod_\calU Q_i \colon Y \to X^\calU$. Moreover, we define $U_T \coloneqq \prod_\calU \tilde{U}_{i,T} \in \calL(Y)$ for each $T \in \bigcup_{i \in I} \mathcal{T}_i$, where $\tilde{U}_{i,T} = U_{i,T}$ if $T \in \calT_i$ and $\tilde{U}_{i,T} = \Id_{Y_i}$ otherwise. Then the diagram
	\begin{center}
		\begin{tikzcd}
			Y \arrow{rrr}{U_{T_1} \dots U_{T_n}} & & & Y \arrow{d}{Q} \\
			X^\calU \arrow{u}{J} \arrow{rrr}{(T_1 \cdots T_n)^\calU} & & & X^\calU \arrow{d} \\
			X \arrow{u} \arrow{rrr}{T_1 \cdots T_n} & & & X
		\end{tikzcd}
	\end{center}
	commutes for every $n \in \bbN_0$ and for all $T_1,\dots,T_n \in \cup_{i \in I} \calT_i$; here, the mapping $X \to X^\calU$ denotes the canonical injection and $X^\calU \to X$ the mapping induced by the weak limit along $\calU$ (which exists since $X$ is reflexive). This proves the assertion.
\end{proof}

\begin{proof}[Proof of Proposition~\ref{prop:finitary}]
	The implication ``(i) $\Rightarrow$ (ii)'' is obvious. So assume that~(ii) holds. If we denote the family of all finite subsets of $\calT$ by $\mathfrak{F}$ and apply Lemma~\ref{lem:extend-dilations-along-a-net} to the net $(\mathcal{F})_{\mathcal{F} \in \mathfrak{F}}$, we obtain~(i).
\end{proof}

\subsection*{\texorpdfstring{$N$}{N}-dilations}

For an operator $T \in \mathcal{L}(X)$, consider the dilation equality
\begin{align*}
	T^n = Q \, U^n J
\end{align*}
from Definition~\ref{def:dilations}(a). For a dilation, we need a Banach space $Y$ and operators $J$, $Q$ and $U$ for which the  equality is satisfied for all $n \in \bbN_0$. However, we shall see in Sections~\ref{sec:convex-combinations} and~\ref{sec:convex-combinations-simultaneous} that it is much easier to achieve this for the first $N$ powers only. In this subsection we show that this is, in a sense, sufficient to obtain a dilation of $T$.

\begin{definition} \label{def:N-dilations}
	Let $\calX$ be a class of Banach spaces, let $X \in \calX$ and $N \in \bbN$.
	\begin{enumerate}
		\item[(a)] An operator $T \in \calL(X)$ has a \emph{$N$-dilation} in $\calX$ if there exist a space $Y \in \calX$, linear contractions $J\colon X \to Y$, $Q\colon Y \to X$ and a linear invertible isometry $U \in \calL(Y)$ such that for all $n \in \{0,\dots,N\}$
		\begin{align*}
			T^n = QU^nJ.
		\end{align*}
		\item[(b)] A set of operators $\calT \subseteq \calL(X)$ has a \emph{simultaneous $N$-dilation} in $\calX$ if there exist a space $Y \in \calX$, linear contractions $J\colon X \to Y$, $Q\colon Y \to X$ and invertible isometries $U_T \in \calL(Y)$ (for $T \in \calT$) such that the equality
		\begin{align*}
			T_1 \dots T_n = Q \, U_{T_1} \dots U_{T_n} \, J
		\end{align*}
		holds for all $n \in \{0,\dots,N\}$ and all $T_1,\dots, T_n \in \calT$.
	\end{enumerate}
\end{definition}

Note that, if $T$ or $\calT$ has a (simultaneous) $N$-dilation in $\calX$, then it also has a (simultaneous) $M$-dilation in $\calX$ for every $M \le N$. 

For single operators on Hilbert spaces, $N$-dilations have been studied in various contexts in the literature, for instance in \cite{Neunzert1963}, \cite{Thompson1982}, \cite[Section~3]{Nagy2013} and~\cite[Sections~1--3]{Levy2014}. Moreover, there is a concept that one might call a \emph{commutative simultaneous $N$-dilation} of a set of commuting operators; this has also been studied on finite dimensional Hilbert spaces, see e.g.~\cite{McCarthy2013} and \cite[Sections~4--5]{Levy2014}. 

For our purposes, $N$-dilations are essential for the construction of dilations for convex combinations of given operators in Sections~\ref{sec:convex-combinations} and~\ref{sec:convex-combinations-simultaneous}. The following proposition shows that, in order to construct a (simultaneous) dilation, it is actually sufficient to construct a (simultaneous) $N$-dilation for each $N$.

\begin{proposition} \label{prop:N-dilations}
	Fix $p \in (1,\infty)$ and let $\calX$ be a class of Banach spaces which fulfils Assumptions~\ref{ass:framework}. Let $X \in \calX$ and consider an operator $T \in \calL(X)$ and a set of operators $\calT \subseteq \calL(X)$.
	\begin{enumerate}
		\item[(a)] If $T$ has an $N$-dilation in $\calX$ for each $N \in \bbN$, then $T$ has a dilation in $\calX$.
		\item[(b)] If $\calT$ has a simultaneous $N$-dilation in $\calX$ for each $N \in \bbN$, then $\calT$ has a simultaneous dilation in $\calX$.
	\end{enumerate}
\end{proposition}
\begin{proof}
	The proof is not too different from the proof of Lemma~\ref{lem:extend-dilations-along-a-net}. First note that (a) follows from~(b) by setting $\calT = \{T\}$, so it suffices to prove~(b). By Assumptions~\ref{ass:framework} there exist, for every $N \in \bbN$, a space $Y_N \in \calL(X)$, linear contractions $J_N\colon X \to Y_N$ and $Q_N\colon Y_N \to X$ and linear invertible isometries $U_{N,T} \in \calL(Y_N)$ (for $T \in \calT$) such that for $n \in \{0,\dots,N\}$ and all $T_1,\dots,T_n \in \calT$ one has the dilation equality
	\begin{align*}
		T_1 \cdots T_n  = Q_N \,  U_{N,T_1} \cdots U_{N,T_n} J_N.
	\end{align*}
	
	Fix a free ultrafilter $\calU$ on $\bbN$; we define $Y \coloneqq \prod_{\calU} Y_N$, $J \coloneqq \prod_\calU J_N \colon X^\calU \to Y$ and $Q \coloneqq \prod_\calU Q_N \colon Y \to X^\calU$. 
	Also set $U_T \coloneqq \prod_\calU U_{T,N}$ for every $T \in \calT$. The operators $U_T$ are invertible isometries on $Y$ and according to Assumptions~\ref{ass:framework} we have $Y \in \calX$. Moreover, the following diagram commutes for every $n \in \bbN_0$ and all $T_1,\dots,T_n \in \calT$:
	\begin{center}
		\begin{tikzcd}
			Y \arrow{rrr}{U_{T_1} \cdots U_{T_n}} & & & Y \arrow{d}{Q} \\
			X^\calU \arrow{u}{J} \arrow{rrr}{(T_1 \cdots T_n)^\calU} & & & X^\calU \arrow{d} \\
			X \arrow{u} \arrow{rrr}{T_1 \cdots T_n} & & & X
		\end{tikzcd}
	\end{center}
	Here, $X \to X^\calU$ denotes the canonical embedding of $X$ into its ultrapower and $X^\calU \to X$ denotes the operator induced by the weak limit along $\calU$ (which exists since $X$ is reflexive). The diagram shows that $\calT$ has a simultaneous dilation.
\end{proof}

\section{Elementary properties of simultaneous dilations} \label{sec:elementary-properties}

\subsection*{Products of dilating operators} If two operators $T$ and $S$ on a Banach space $X$ have a dilation in a class of Banach spaces $\calX$, then it is by no means clear whether the product $ST$ has a dilation in $\calX$, too. 
If however $T$ and $S$ have a simultaneous dilation, then obviously so does $ST$. Let us note this -- in a slightly more general form -- in the following proposition.

\begin{proposition} \label{prop:product-of-dilating-operators}
	Let $\calX$ be a class of Banach spaces, let $X \in \calX$ and assume that $\calT \subseteq \calL(X)$ has a simultaneous dilation in $\calX$. Then the \emph{multiplicative semigroup generated by $\calT$}, i.e.\ the set $\{T_1 \cdots T_n: \, n \in \bbN, \; T_1,\dots,T_n \in \calT\}$, has a simultaneous dilation in $\calX$, too (and in fact, to the same space as $\calT$).
\end{proposition}

\subsection*{Strong operator limits} The next proposition shows that simultaneous dilations are preserved by strong operator limits; this proves the topological part of Theorem~\ref{thm:main-result}.

\begin{proposition} \label{prop:stop-limits}
	Fix $p \in (1,\infty)$ and let $\calX$ be a class of Banach spaces which fulfills Assumptions~\ref{ass:framework}. Let $X \in \calX$ and let $\calT \subseteq \calL(X)$ be a set of operators which has a simultaneous dilation in $\calX$. Then the strong operator closure of $\calT$ has a simultaneous dilation in $\calX$, too.
\end{proposition}
\begin{proof}
	There exist a space $Y \in \calL(X)$, linear contractions $J\colon X \to Y$ and $Q\colon Y \to X$ and invertible isometries $U_T \in \calL(Y)$ (for $T \in \calT$) such that the equality $\prod_{k=1}^n T_k = Q\, \prod_{k=1}^n U_{T_k} \; J$ holds for each $n \in \bbN$ and all $T_1,\dots,T_n \in \calT$.
	
	Let $\calS$ denote the strong operator closure of $\calT$ in $\calL(X)$. We can find a directed set $I$ such that, for each $S \in \calS$, there exists a net $(T_{i,S})_{i \in I}$ in $\calT$ that converges strongly to $S$; for instance, we can take $I$ to be the neighborhood filter of $0$ in $\calL(X)$ with respect to the strong operator topology (endowed with the canonical order). Choose an ultrafilter $\calU$ on $I$ which contains the filter base
	\begin{align*}
		\big\{ \{j \in I: \, j \ge i\} : \; i \in I \big\}.
	\end{align*}
	Let $\tilde J \colon Y \to Y^\calU$ be the canonical injection and let $\tilde Q \colon Y^\calU \to Y$ be the operator induced by the weak limit along $\calU$ (which exists since $Y$ is reflexive). For each $S \in \calS$ we choose a net $(T_{i,S})_{i \in I}$ that converges strongly to $S$ and we define $\tilde U_S \coloneqq \prod_\calU U_{T_{i,S}} \in \calL(Y^\calU)$. 
	
	Now, let $n \in \bbN$ and $S_1,\dots,S_n \in \calS$. We show that $S_1 \cdots S_n = Q \tilde Q\tilde U_{S_1} \cdots \tilde U_{S_n} \tilde J J$. Indeed, we have on the one hand
	\begin{align*}
		\prod_{k=1}^n S_k & = \prod_{k=1}^n \stoplim_{i \to \calU} T_{i,S_k} = \stoplim_{i \to \calU} \prod_{k=1}^n T_{i,S_k} = \stoplim_{i \to \calU} Q \, \prod_{k=1}^n U_{T_{i,S_k}} \; J,
	\end{align*}
	where $\stoplim$ denotes the limit with respect to the strong operator topology; the first equality follows from the choice of the ultrafilter $\calU$ and for the second equality we used that operator multiplication is jointly continuous with respect to the strong operator topology on bounded sets. On the other hand, we have for every $x \in X$
	\begin{align*}
		Q \tilde Q \, \prod_{k=1}^n \tilde U_{S_k} \; \tilde J Jx & = Q \tilde Q \big(\prod_{k=1}^n U_{T_{i,S_k}} \; Jx\big)_\calU \\
		& = Q \wlim_{i \to \calU} \big( \prod_{k=1}^n U_{T_{i,S_k}} \; Jx \big) = \wlim_{i \to \calU} \big( Q \, \prod_{k=1}^n U_{T_{i,S_k}} \; Jx \big),
	\end{align*}
	where $\wlim$ denotes the weak limit. This proves the assertion.
\end{proof}

We point out that the above proof does not work for the weak operator closure of $\calT$ since operator multiplication is in general not jointly continuous with respect to the weak operator topology (not even on bounded sets of operators). However, one can show the following result by a similar technique as in the above proof: let $(T_i)_{i \in I}$ be a net in $\calL(X)$, let $T \in \calL(X)$ and assume that, for each power $n \in \bbN$, the net $\big((T_i)^n\big)_{i \in I}$ converges to $T^n$ with respect to the weak operator topology. If each operator $T_i$ has a dilation in $\calX$ (not necessarily a simultaneous dilation for all $T_i$), then $T$ has a dilation in $\calX$, too (provided that $\calX$ fulfills Assumptions~\ref{ass:framework}). See also~\cite{Pel81} and the discussion in Remark~\ref{rem:peller-power-approximation-vs-our-approach}.

\subsection*{The zero operator} We can add the zero operator to a given set of simultaneously dilating operators. In conjunction with Theorem~\ref{thm:convex-combinations-simultaneous} this shows that simultaneous dilations are stable with respect to \emph{subconvex combinations}, i.e.\ linear combinations with positive coefficients adding up to at most $1$; see Remark~\ref{rem:subconvex-and-absolutely-convex-hull}(b).

\begin{proposition} \label{prop:zero-operator}
	Fix $p \in (1,\infty)$ and let $\calX$ be a class of Banach spaces which fulfills Assumptions~\ref{ass:framework}. Let $X \in \calX$ and assume that $\calT \subseteq \calL(X)$ has a simultaneous dilation in $\calX$. Then $\calT \cup \{0\}$ has a simultaneous dilation in $\calX$.
\end{proposition}
\begin{proof}
	By assumption we can find a space $Y \in \calX$, linear contractions $J\colon X \to Y$ and $Q\colon Y \to X$ and invertible isometries $U_T \in \calL(Y)$ (for $T \in \calT$) such that
	\begin{align*}
		T_1 \cdots T_n = Q \, U_{T_1} \cdots U_{T_n} J
	\end{align*}
	for every $n \in \bbN_0$ and all $T_1,\dots, T_n \in \calT$. Define $U_0 \coloneqq 0 \in \calL(Y)$ (which is of course neither invertible, nor an isometry). It suffices to show that the set of operators $\calU \coloneqq \big\{U_T: T \in \calT \cup \{0\}\big\}$ has a simultaneous dilation in $\calX$ because then the diagram
	\begin{center}
		\begin{tikzcd}
			\tilde Y \arrow{rrr}{\tilde V_{U_{T_1}} \cdots \tilde V_{U_{T_n}}} & & & \tilde Y \arrow{d}{\tilde Q} \\
			Y \arrow{rrr}{U_{T_1} \cdots U_{T_n}} \arrow{u}{\tilde J} & & & Y \arrow{d}{Q} \\
			X \arrow{u}{J} \arrow{rrr}{T_1 \cdots T_n} & & & X
		\end{tikzcd}
	\end{center}
	commutes for an appropriate choice of $\tilde Y$, $\tilde J$, $\tilde Q$ and $V_U \in \calL(\tilde Y)$ (for $U \in \calU$) and for all $n \in \bbN_0$ and $T_1, \dots, T_n \in \calT \cup \{0\}$.
	
	Fix an arbitrary $N \in \IN$. In order to show that $\big\{U_T: T \in \calT \cup \{0\}\big\}$ has a simultaneous dilation in $\calX$ it suffices, according to Proposition~\ref{prop:N-dilations}, to construct a simultaneous $N$-dilation for $\calU$ in $\calX$. To this end, choose $\tilde Y \coloneqq \ell^p_{N+1}(Y) \in \calX$ and set $\tilde Jx = (x,0,\dots,0) \in \tilde Y$ for all $x \in Y$; for all $(x_1,\dots,x_{N+1}) \in \tilde Y$ we define $\tilde Q(x_1,\dots,x_{N+1}) = x_1 $ and 
	\begin{align*}
		V_U(x_1,\dots,x_{N+1}) = 
		\begin{cases}
			(U x_1,\dots,U x_{N+1}) \quad & \text{if } U \not= 0, \\
			(x_2,\dots,x_{N+1},x_1) \quad & \text{if } U = 0.
		\end{cases}
	\end{align*}
	Then one easily checks that for all $n \in \{0,\dots, N\}$ and all $U_1,\dots, U_n \in \calU$
	\begin{align*}
		U_1 \cdots U_n = \tilde Q \, V_{U_1} \dots V_{U_n} \tilde J
	\end{align*}
	Hence, $\calU$ indeed has a simultaneous $N$-dilation in $\calX$.
\end{proof}

\begin{remark}
	It is worthwhile pointing out that the reduction to $N$-dilations used in the proof of Proposition~\ref{prop:zero-operator} (by employing Proposition~\ref{prop:N-dilations}) is not as essential as it might seem at first glance. Indeed, if we add to Assumptions~\ref{ass:framework} the mild condition that, for each $X \in \calL(X)$, the vector-valued $\ell^p$-space $\ell^p(X) \coloneqq \ell^p(\bbZ;X)$ is contained in $\calX$, then we can define $Y \coloneqq \ell^p(\bbZ;X)$ and choose $V_0$ to be the left (or right) shift operator and immediately obtain a simultaneous dilation instead of only a simultaneous $N$-dilation of the set $\calU$ (where $J$, $Q$ and $V_U$ for $U \not= 0$ should be defined similarly as in the above proof of Proposition~\ref{prop:zero-operator}). Nevertheless, $N$-dilations and Proposition~\ref{prop:N-dilations} will be an indispensable tool for us in Sections~\ref{sec:convex-combinations} and~\ref{sec:convex-combinations-simultaneous}.
\end{remark}

\section{Dilation of convex combinations} \label{sec:convex-combinations}

The goal of this section is to prove the following theorem.

\begin{theorem} \label{thm:convex-combination}
	Fix $p \in (1,\infty)$ and let $\calX$ be a class of Banach spaces which fulfills Assumptions~\ref{ass:framework}. Let $X \in \calX$ and assume that $\calT \subseteq \calL(X)$ has a simultaneous dilation in $\calX$. If $T$ is in the convex hull of $\calT$, then $T$ has a dilation in $\calX$.
\end{theorem}

In Section~\ref{sec:convex-combinations-simultaneous} we show the more general result that the convex hull of $\calT$ has a simultaneous dilation in $\calX$, which, in conjunction with Proposition~\ref{prop:stop-limits}, proves our main result Theorem~\ref{thm:main-result}. The next result, which is not immediate to prove with bare hands, is a direct consequence of Theorem~\ref{thm:convex-combination} and Proposition~\ref{prop:zero-operator}.

\begin{corollary} \label{cor:multiple-of-a-single-operator}
	Fix $p \in (1,\infty)$ and let $\calX$ be a class of Banach spaces which fulfills Assumptions~\ref{ass:framework}. Let $X \in \calX$ and assume that $T \in \calL(X)$ has a dilation in $\calX$. Then, for every $\lambda \in [0,1]$, the operator $\lambda T$ has a dilation in $\calX$.
\end{corollary}

It is a consequence of Proposition~\ref{prop:zero-operator} and the more general Theorem~\ref{thm:convex-combinations-simultaneous} below that $\{\lambda T: \; \lambda \in [0,1]\}$ even has a simultaneous dilation in $\calX$.

\subsection*{Main ideas}

Before we give the proof of Theorem~\ref{thm:convex-combination}, we explain some of the main ideas. According to Proposition~\ref{prop:N-dilations} it suffices to show that $T$ has an $N$-dilation in $\calX$ for every $N \in \bbN$. To give the reader an idea of how this can be accomplished, let us first consider the following special case of Theorem~\ref{thm:convex-combination}: let $T_1,T_2 \in \calL(X)$ be two invertible isometries and let $\lambda_1,\lambda_2 \in [0,1]$ with $\lambda_1 + \lambda_2 = 1$. We want to show that $T \coloneqq \lambda_1T_1 + \lambda_2T_2$ has an $N$-dilation in $\calX$ for $N \in \bbN$. 

For $N = 1$ this can  be accomplished as follows. Set $Y = \ell_2^p(X) = X^2$ and define
\begin{align}
	\label{eq:1-dilation-of-convex-combination}
	U \coloneqq
	\begin{pmatrix}
		T_1 & 0 \\
		0 & T_2
	\end{pmatrix} \in \mathcal{L}(Y).
\end{align}
This is obviously an invertible isometry on $Y$ since $T_1$ and $T_2$ were assumed to be invertible isometries on $X$. Moreover, we define $J\colon X \to Y$ and $Q\colon Y \to X$ by
\begin{align*}
	Jx =
	\begin{pmatrix}
		\lambda_1^{1/p}x \\ \lambda_2^{1/p}x
	\end{pmatrix},
	\qquad
	Q
	\begin{pmatrix}
		x_1 \\ x_2
	\end{pmatrix}
	= 
	\lambda_1^{1/q}x_1 + \lambda_2^{1/q}x_2
\end{align*}
for all $x,x_1,x_2 \in X$, where $q$ is the Hölder conjugate of $p$, i.e.\ $1/p + 1/q = 1$. Clearly, $J$ is an isometry and $Q$ is a contraction by Hölder's inequality. Moreover, for $x \in X$
\begin{align*}
	Q \, U^0 Jx = QJx = \lambda_1^{1/q}\lambda_1^{1/p}x + \lambda_2^{1/q}\lambda_2^{1/p}x = x = T^0 x
\end{align*}
and
\begin{align*}
	Q \, U^1 J x = 
	Q
	\begin{pmatrix}
		\lambda_1^{1/p} T_1 x \\
		\lambda_2^{1/p} T_2 x
	\end{pmatrix}
	= \lambda_1 T_1x + \lambda_2 T_2x = Tx.
\end{align*}
Hence, we have constructed a $1$-dilation of $T$ in $\calX$. In order construct a $2$-dilation, one can proceed as follows. Let $Y = \ell^p_4(X) = X^4$ and define $U \in \calL(Y)$ by
\begin{align}
	\label{eq:2-dilation-of-convex-combination}
	U =
	\begin{pmatrix}
		T_1 &     &     &     \\
		    & T_2 &     &     \\
		    &     &     & T_1 \\
		    &     & T_2 & 
	\end{pmatrix},
\end{align}
where the empty entries are understood to be the zero operator. Moreover, we define $J\colon X \to Y$ and $Q\colon Y \to X$ by
\begin{align*}
	Jx = ((\lambda_1 \lambda_1)^{1/p}x,
		(\lambda_2 \lambda_2)^{1/p}x,
		(\lambda_1 \lambda_2)^{1/p}x,
		(\lambda_2 \lambda_1)^{1/p}x)^T
\end{align*}
for all $x \in X$ and
\begin{align*}
	Q(x_1, \ldots, x_n)^T = (\lambda_1 \lambda_1)^{1/q}x_1 + (\lambda_2 \lambda_2)^{1/q}x_2 + (\lambda_1 \lambda_2)^{1/q}x_2 + (\lambda_2 \lambda_1)^{1/q}x_4
\end{align*}
for all $x_1,\dots,x_4 \in X$. Again, $J$ and $Q$ are contractions and $U$ is an invertible isometry. Moreover, it is easy to check that $Q\, U^n J = T^n$ for all $n \in \{0,1,2\}$. We have thus constructed a $2$-dilation of $T$ in $\calX$.

At first glance, this is where the story ends, since there is no obvious way to generalize the above constructions in order to obtain a $3$-dilation of $T$. Indeed, formulas~\eqref{eq:1-dilation-of-convex-combination} and~\eqref{eq:2-dilation-of-convex-combination} suggest that, in order to construct an $N$-dilation of $T$, we should choose $Y = X^{2^N}$ and define $U$ as some kind of permutation matrix whose entries are the operators $T_1$ and $T_2$ instead of ones. Yet, it does not seem to be clear what permutation matrix we should choose and which entries shall be chosen to be $T_1$ and which to be $T_2$. This suggests that we should make some changes to the above construction in order obtain a structure which also works for $N$-dilations.

To this end, we proceed as follows. Fix $N \in \bbN$ and consider the $N$-cycle $\sigma \coloneqq (1 \dots N)$, i.e.\ the permutation on the set $\{1,\dots,N\}$ that maps $1$ to $2$, $2$ to $3$, \dots, and $N$ to $1$. We intend to take the $N \times N$-permutation matrix induced by $\sigma$ and to replace the ones in the matrix with the operators $T_1$ and $T_2$. Since there are only two operators $T_1$ and $T_2$ but $N$ entries that we have to replace, there is no canonical choice of which entry should become which operator. We thus follow the folklore role of thumb that, if there is no canonical choice to make, then it is best to make all possible choices simultaneously: let $\calA$ denote the set of all mappings from $\{1,\dots,N\}$ to $\{1,2\}$. For each $\alpha \in \calA$ we consider the sequence of $N$ operators $T_{\alpha(1)}, \dots, T_{\alpha(N)}$ and replace the ones in the permutation matrix induced by $\sigma$ with those operators; this yields an invertible isometry $U_\alpha \in \calL(X^{N})$. Finally, we define $U \coloneqq \bigoplus_{\alpha \in \calA} U_\alpha \in \calL(Y)$ where $Y \coloneqq X^{N \lvert \calA \rvert} = X^{N 2^N}$.

It turns out that, with appropriate choices of $J$ and $Q$, this really yields an $N$-dilation of $T$, though a few cumbersome computations are needed to verify this. In the next subsection we make the above construction precise and give all necessary details. 
Note that we have to consider convex combinations of finitely many instead of only two operators.

We should point out once again that, for $N = 2$, the construction that we have just described is more complex than the construction of $U$ in~\eqref{eq:2-dilation-of-convex-combination}: we now obtain a $2$-dilation on the space $X^{2 \cdot 2^2} = X^8$ instead of $X^4$.

\subsection*{The proof of Theorem~\ref{thm:convex-combination}}

After we have just explained the main ideas behind the proof of Theorem~\ref{thm:convex-combination}, we are now going to give the proof in detail.

\begin{proof}[Proof of Theorem~\ref{thm:convex-combination}]
Since $\calT$ admits a simultaneous dilation in $\calX$ we may, and will, assume that $\calT$ consists of invertible isometries. There exist $m \in \bbN$, operators $T_1, \dots, T_m \in \calT$ and numbers $\lambda_1,\dots,\lambda_m \in [0,1]$ such that $\sum_{k=1}^m \lambda_k = 1$ and $\sum_{k=1}^m \lambda_k T_k = T$.

Fix $N \in \bbN$. According to Proposition~\ref{prop:N-dilations}(a) it suffices to show that $T$ possesses an $N$-dilation in $\calX$. To this end, let $\calA$ denote the set of all mappings from $\{1,\dots,N\}$ to $\{1,\dots,m\}$. It is convenient to abbreviate $\lambda \coloneqq (\lambda_1, \dots, \lambda_m)$ and to define $\lvert \lambda \rvert_\alpha \coloneqq \prod_{k=1}^N \lambda_{\alpha(k)}$ for each $\alpha \in \calA$. Note that
\begin{align}
	\label{eq:sum-of-all-lambda-alpha}
	\sum_{\alpha \in \calA} \lvert \lambda \rvert_\alpha = \big(\sum_{k=1}^m \lambda_k \big)^N = 1.
\end{align}
Now, we define the space $Y$ and the mappings $J\colon X \to Y$ and $Q\colon Y \to X$. We set
\begin{align*}
	Y \coloneqq \ell^p_{N m^N}(X) = \ell^p_{N \lvert \calA \rvert}(X) = (X^N)^{\calA},
\end{align*}
where the latter space is endowed with the $p$-norm. For each $\alpha \in \calA$ we set
\begin{align*}
	J_\alpha\colon X \to X^N, \qquad J_\alpha x = \big(\frac{\lvert \lambda\rvert_\alpha}{N}\big)^{1/p} \big( x, \dots, x\big)
\end{align*}
and we define $J \colon X \to Y$ by $Jx = (J_\alpha x)_{\alpha \in \calA}$ for all $x \in X$. It readily follows from formula~\eqref{eq:sum-of-all-lambda-alpha} that $J$ is isometric. Moreover, for each $y = (x_{k,\alpha})_{k \in \{1,\dots,N\}, \, \alpha \in \calA} \in Y = (X^N)^{\calA}$ we define
\begin{align*}
	Qy \coloneqq \sum_{\alpha \in \calA} \big(\frac{\lvert \lambda \rvert_\alpha}{N}\big)^{1/q} \sum_{k=1}^N x_{k,\alpha},
\end{align*}
where $q \in (1,\infty)$ denotes the conjugate index of $p$, meaning that $1/p + 1/q = 1$. It follows from H\"older's inequality that $Q$ is contractive. Indeed, if $y$ is as above, then
\begin{align*}
	\lVert Qy\rVert & \le \sum_{\alpha \in \calA} \sum_{k=1}^N \big( \frac{|\lambda|_\alpha}{N} \big)^{1/q} \lVert x_{k,\alpha}\rVert \le \\
	& \le \big( \sum_{\alpha \in \calA} \sum_{k=1}^N \frac{|\lambda|_\alpha}{N} \, \big)^{1/q} \cdot \big( \sum_{\alpha \in \calA} \sum_{k=1}^N \lVert x_{k,\alpha}\rVert^p \, \big)^{1/p} = \lVert y \rVert. 
\end{align*}
Finally, we have to construct an invertible isometry $U \in \calL(Y)$ such that $Q \, U^n J = T^n$ for all $n \in \{0,\dots,N\}$. For each $\alpha \in \calA$ we define $U_\alpha\colon X^N \to X^N$ by
\begin{align*}
	U_\alpha (x_k)_{k \in \{1,\dots,N\}} = \big( T_{\alpha(k)}x_{\sigma(k)} \big)_{k \in \{1,\dots,N\}}
\end{align*}
for every $(x_k)_{k \in \{1,\dots,N\}} \in X^N$; as noted above, $\sigma\colon \{1,\dots,N\} \to \{1,\dots,N\}$ denotes the $N$-cycle $(1 \dots N)$. We point out that $U_\alpha$ can be written in matrix form as
\begin{align*}
	U_\alpha = 
	\begin{pmatrix}
		              & T_{\alpha(1)} &               &        &                 \\
		              &               & T_{\alpha(2)} &        &                 \\
		              &               &               & \ddots &                 \\ 
		              &               &               &        & T_{\alpha(N-1)} \\
		T_{\alpha(N)} &               &               &        &   
	\end{pmatrix}.
\end{align*}
We set $U \coloneqq \bigoplus_{\alpha \in \calA} U_\alpha\colon Y \to Y$. Clearly, $U$ is an invertible isometry since every operator $T_1,\dots,T_m$ is assumed to be an invertible isometry. The only remaining point is to prove that $U$ fulfills the equality $Q \, U^n J = T^n$ for all $n \in \{0,\dots,N\}$.

First observe that we can explicitly compute the powers $(U_\alpha)^n$ of $U_\alpha$ for $n \in \{0,\dots,N\}$. Indeed, for each such $n$, every $\alpha \in \calA$ and every $(x_k)_{k \in \{1,\dots,N\}} \in X^N$ we have
\begin{equation}
	\label{eq:formula_power_single_convex}
	\begin{split}
		(U_\alpha)^n (x_k)_{k \in \{1,\dots,N\}} & = \big( \prod_{j=1}^n T_{\alpha(\sigma^{j-1}(k))} \; x_{\sigma^n(k)} \big)_{k \in \{1,\dots,N\}} \\
		                                         & = \big( \prod_{j=1}^n T_{\alpha(\sigma^{k-1}(j))} \; x_{\sigma^n(k)} \big)_{k \in \{1,\dots,N\}}.
	\end{split}
\end{equation}
The first equality can be seen by induction over $n$ (and holds for all $n \in \bbN_0$). The second equality follows from that fact that $\sigma^{j-1}(k) = \sigma^{k-1}(j)$ for all $j,k \in \{1,\dots,N\}$. Using~\eqref{eq:formula_power_single_convex}, we obtain for $n \in \{0,\dots,N\}$ and $x \in X$
\begin{align*}
	Q \, U^n Jx & = \sum_{\alpha \in \calA} \big(\frac{\lvert \lambda \rvert_\alpha}{N}\big)^{1/q} \sum_{k=1}^N \big(\frac{\lvert \lambda\rvert_\alpha}{N}\big)^{1/p} \big((U_\alpha)^n(x, \dots, x)\big)_k \\
	& = \sum_{\alpha \in \calA}  \frac{\lvert \lambda\rvert_\alpha}{N} \sum_{k=1}^N \prod_{j=1}^n T_{\alpha(\sigma^{k-1}(j))} \; x.
\end{align*}
On the other hand, a short computation shows that
\begin{align*}
	T^n = \sum_{\alpha \in \calA} \lvert \lambda \rvert_\alpha \prod_{k=1}^n T_{\alpha(k)},
\end{align*}
so it only remains to prove the equality
\begin{align}
	\label{eq:critical-equation-for-dilation-equality}
	\sum_{\alpha \in \calA}  \frac{\lvert \lambda\rvert_\alpha}{N} \sum_{k=1}^N \prod_{j=1}^n T_{\alpha(\sigma^{k-1}(j))} = \sum_{\alpha \in \calA} \lvert \lambda\rvert_\alpha \prod_{j=1}^n T_{\alpha(j)}.
\end{align}
To show~\eqref{eq:critical-equation-for-dilation-equality} we need a bit of group theory. Let $S_N$ denote the symmetric group over $N$ elements, i.e.\ the set of all bijections on $\{1,\dots,N\}$, and let $C_N$ denote the cyclic subgroup of $S_N$ generated by the $N$-cycle $\sigma$. The group $C_N$ operates on the set $\calA$ via the group action
\begin{align*}
	C_N \times \calA & \to \calA, \\
	(\tau, \alpha) & \mapsto \alpha \circ \tau.
\end{align*}
We call two elements $\alpha,\beta \in \calA$ \emph{equivalent} and denote this by $\alpha \sim \beta$ if they have the same orbit under this group action, i.e.\ if there exists a permutation $\tau \in C_N$ such that $\alpha = \beta \circ \tau$. Clearly, $\sim$ is an equivalence relation on $\calA$. Note that we have $|\lambda|_\alpha = |\lambda|_\beta$ whenever $\alpha \sim \beta$ (but the converse is of course false). 

Consider a fixed equivalence class $A \subseteq \calA$ of $\sim$. Since, for $\alpha \in A$, the number $\lvert \lambda\rvert_\alpha$ does not depend on the choice of $\alpha$, it suffices to prove
\begin{align}
	\label{eq:critical-equation-for-dilation-equality-2}
	\sum_{\alpha \in A} \sum_{k=1}^N \prod_{j=1}^n T_{\alpha(\sigma^{k-1}(j))} = N \sum_{\alpha \in A} \prod_{j=1}^n T_{\alpha(j)}
\end{align}
in order to show~\eqref{eq:critical-equation-for-dilation-equality}. Fix an $\alpha_0 \in A$. The equivalence class $A$ is given by $A = \{\alpha_0 \circ \tau: \; \tau \in C_N\}$ and thus, we can replace the summation on both sides of~\eqref{eq:critical-equation-for-dilation-equality-2} with a summation over $C_N$. Yet, we have to be a bit careful here since the  surjective mapping $C_N \ni \tau \to \alpha_0 \circ \tau \in A$ might not be injective. To account for this, we use Proposition~\ref{prop:group-operation} from the appendix which tells us that, for each $\alpha \in A$, there exist exactly $N/\lvert A\rvert$ elements $\tau \in C_N$ with $\alpha_0 \circ \tau = \alpha$. Thus, the left hand side of~\eqref{eq:critical-equation-for-dilation-equality-2} becomes
\begin{align*}
	\frac{\lvert A\rvert}{N} \sum_{\tau \in C_N} \sum_{k=1}^N \prod_{j=1}^n T_{(\alpha_0 \circ \tau \circ \sigma^{k-1})(j)} = \frac{\lvert A\rvert}{N} \sum_{\tau,\rho \in C_N} \prod_{j=1}^n T_{(\alpha_0 \circ \tau \circ \rho)(j)}
\end{align*}
and the right hand side of~\eqref{eq:critical-equation-for-dilation-equality-2} becomes
\begin{align*}
	\frac{\lvert A\rvert}{N} \; N \sum_{\tau \in C_N} \prod_{j=1}^n T_{(\alpha_0 \circ \tau)(j)}. 
\end{align*}
The mapping $C_N \times C_N \ni (\tau, \rho) \mapsto \tau \circ \rho \in C_N$ hits each element in $C_N$ exactly $N$ times (see Proposition~\ref{prop:count-group-elements} in the appendix), so both the left and the right side of~\eqref{eq:critical-equation-for-dilation-equality-2} coincide. This proves~\eqref{eq:critical-equation-for-dilation-equality-2}, hence~\eqref{eq:critical-equation-for-dilation-equality} and thus the theorem.
\end{proof}

\section{Simultaneous dilation of convex combinations} \label{sec:convex-combinations-simultaneous}

The following theorem generalizes Theorem~\ref{thm:convex-combination} to simultaneous dilations. 

\begin{theorem} \label{thm:convex-combinations-simultaneous}
	Fix $p \in (1,\infty)$ and let $\calX$ be a class of Banach spaces that fulfills Assumptions~\ref{ass:framework}. Let $X \in \calX$ and assume that $\calT \subseteq \calL(X)$ has a simultaneous dilation in $\calX$. Then the convex hull $\operatorname{conv} \calT$ of $\calT$ has a simultaneous dilation in $\calX$.
\end{theorem}

Our main result, Theorem~\ref{thm:main-result}, is an immediate consequence of Theorem~\ref{thm:convex-combinations-simultaneous}, Proposition~\ref{prop:stop-limits} and of the fact that, for convex sets of operators, the strong and the weak operator closure coincide \cite[Corollary VI.1.5]{DunSch1958}.

The proof of Theorem~\ref{thm:convex-combinations-simultaneous} is similar to the proof of Theorem~\ref{thm:convex-combination}, but technically more involved. The major obstacle is that, in the proof of Theorem~\ref{thm:convex-combination}, the maps $J$ and $Q$ depend on the convex coefficients $\lambda_1,\dots,\lambda_m$. If we want to dilate several operators $T^{(1)},\dots, T^{(r)} \in \conv \calT$ instead of only one operator $T$, we obtain different sets of convex coefficients and thus -- if we want to use the same technique as for Theorem~\ref{thm:convex-combination} -- different maps $J$ and $Q$ for each operator $T^{(1)}, \dots, T^{(r)}$. This contradicts the definition of a simultaneous dilation. Fortunately, there is a trick to circumvent this problem; this trick is explained in the first part of the proof of Theorem~\ref{thm:convex-combinations-simultaneous}. The rest of the proof is quite similar to the proof of Theorem~\ref{thm:convex-combination}.

\begin{proof}[Proof of Theorem~\ref{thm:convex-combinations-simultaneous}]
	Let $\calS$ denote the set of all convex combinations of $\calT$ with rational convex coefficients. 
	Then $\conv \calT$ is contained in the strong operator closure of $\calS$, so it suffices by Proposition~\ref{prop:stop-limits} to show that $\calS$ has a simultaneous dilation in $\calX$. 
	To this end, it suffices in turn to prove that every finite subset $\calF$ of $\calS$ has a simultaneous dilation in $\calX$, see Proposition~\ref{prop:finitary}.
	
	So let $\calF$ be a finite subset of $\calS$. Since every operator in $\calF$ can be written as a convex combination of finitely many operators in $\calT$ with rational convex coefficients we can find a number $m \in \bbN$ and, for each $F \in \calF$, operators $T_{1,F},\dots, T_{m,F} \in \calT$ such that
	\begin{align}
		\label{eq:multiple-convex-combinations}
		F = \sum_{k=1}^m \frac{1}{m} T_{k,F}.
	\end{align}
	Since the operators $T_{k,F}$ (for $k \in \{1,\dots, m\}$ and $F \in \calF$) have a simultaneous dilation in $\calX$ we may, and will, assume from now on that each operator $T_{k,F}$ is an invertible isometry. The point of the above manipulations is that we have represented the operators in $F \in \calF$, which we wish to dilate, with the same convex coefficients for each $F$.
	
	Fix $N \in \bbN$. We show that $\calF$ has a simultaneous dilation in $\calX$, and to this end it suffices due to Proposition~\ref{prop:N-dilations} to prove that $\calF$ has a simultaneous $N$-dilation in $\calX$. We first construct the space $Y$ and the mappings $J \colon X \to Y$ and $Q \colon Y \to X$ used in Definition~\ref{def:N-dilations}(b). As in the proof of Theorem~\ref{thm:convex-combination} we denote the set of all mappings from $\{1,\dots,N\}$ to $\{1,\dots,m\}$ by $\calA$ and we let 
	\begin{align*}
		Y \coloneqq \ell^p_{N m^N}(X) = \ell^p_{N \lvert \calA \rvert}(X) = X^{N m^N} = (X^N)^{\calA},
	\end{align*}
	Also analogously to the proof of Theorem~\ref{thm:convex-combination} we define
	\begin{align*}
		J\colon X \to Y, \qquad Jx = \big(\frac{1}{N m^N} \big)^{1/p} \big( x, \dots, x\big)
	\end{align*}
	and
	\begin{align*}
		Q\colon Y \to X, \qquad Q (x_{k,\alpha})_{k \in \{1,\dots,N\}, \, \alpha \in \calA} = \big( \frac{1}{N m^N} \big)^{1/q} \sum_{\alpha \in \calA} \sum_{k=1}^N x_{k,\alpha},
	\end{align*}
	where $q$ is the Hölder conjugate of $p$. Then $J$ is isometric and $Q$ is contractive.
	
	Finally, we construct $(U_{F})_{F \in \calF}$ in a similar way as we defined the operator $U$ in the proof of Theorem~\ref{thm:convex-combination}. For $F \in \calF$ and $\alpha \in \calA$ we define $U_{\alpha, F} \in \calL(X^N)$ by
	\begin{align*}
		U_{\alpha,F} (x_k)_{k \in \{1,\dots,N\}} = \big( T_{\alpha(k),F}x_{\sigma(k)} \big)_{k \in \{1,\dots,N\}}
	\end{align*}
	for all $(x_k)_{k \in \{1,\dots,N\}}$ and we set $U_F \coloneqq \bigoplus_{\alpha \in \calA} U_{\alpha,F} \in \calL(Y)$ for each $F \in \calF$. Clearly, each operator $U_F$ is an invertible isometry on $Y$ and we only have to verify 
	\begin{align*}
		Q \, U_{F_1} \cdots U_{F_n} J = F_1 \cdots F_n
	\end{align*}
	for each $n \in \{0,\dots,N\}$ and all $F_1, \dots, F_n \in \calF$. So fix $n \in \{0,\dots,N\}$ and $F_1,\dots,F_n \in \calF$.
	We note that, similarly as in the proof of Theorem~\ref{thm:convex-combination}, the formula
	\begin{align*}
		\begin{split}
			\prod_{j=1}^n U_{\alpha,F_j} \; (x_k)_{k \in \{1,\dots,N\}} & = \big( \prod_{j=1}^n T_{\alpha(\sigma^{j-1}(k)), F_j} \; x_{\sigma^n(k)} \big)_{k \in \{1,\dots,N\}} \\
		                                         & = \big( \prod_{j=1}^n T_{\alpha(\sigma^{k-1}(j)), F_j} \; x_{\sigma^n(k)} \big)_{k \in \{1,\dots,N\}}
		\end{split}
	\end{align*}
	holds for all $\alpha \in \calA$ and all $(x_k)_{k \in \{1,\dots,N\}} \in X^N$. Using this, we obtain for each $x \in X$
	\begin{align*}
		Q \, U_{F_1} \cdots U_{F_n} J x = \frac{1}{N m^N} \sum_{\alpha \in \calA} \sum_{k=1}^N \prod_{j=1}^n T_{\alpha(\sigma^{k-1}(j)), F_j} \; x.
	\end{align*}
	On the other hand, using~\eqref{eq:multiple-convex-combinations} and the fact that $n \le N$, one easily checks that
	\begin{align*}
		F_1 \cdots F_n x = \frac{1}{m^N} \sum_{\alpha \in \calA} \prod_{j=1}^n T_{\alpha(j), F_j} \; x.
	\end{align*}
	for each $x \in X$. Hence, we merely have to prove that the right hands sides of the previous two equations coincide. We denote by $\sim$ the same equivalence relation on $\calA$ as in the proof of Theorem~\ref{thm:convex-combination} and we fix an equivalence class $A \subseteq \calA$. To conclude the proof, it suffices to show that
	\begin{align*}
		\sum_{\alpha \in A} \sum_{k=1}^N \prod_{j=1}^n T_{\alpha(\sigma^{k-1}(j)), F_j} = N \sum_{\alpha \in A} \prod_{j=1}^n T_{\alpha(j), F_j}.
	\end{align*}
	This follows by exactly the same arguments as in the proof of Theorem~\ref{thm:convex-combination}.
\end{proof}

\section{Application to Banach spaces with super properties} \label{sec:super-properties}

In this section we apply our approach to classes of Banach spaces which fulfill certain regularity properties. Let us begin with the class of super-reflexive Banach spaces. This class is not stable with respect to ultra-products and thus, it does not fulfill Assumptions~\ref{ass:framework}. Nevertheless, we can apply our theory by employing Proposition~\ref{prop:X_Z_satisfies_assumptions}.

\begin{theorem} \label{thm:super-reflexive-spaces}
	Let $Z$ be a super-reflexive Banach space and let $\calT$ denote the weakly closed convex hull of all invertible isometries in $\calL(Z)$. Then $\calT$ has a simultaneous dilation in the class of all super-reflexive Banach spaces.
\end{theorem}
\begin{proof}
	Let $\calX_Z$ be the class of Banach spaces defined in Proposition~\ref{prop:X_Z_satisfies_assumptions}. Then, according to this proposition, $\calX_Z$ contains the space $Z$ and fulfills the Assumptions~\ref{ass:framework}. Hence, $\calT$ has a simultaneous dilation in $\calX_Z$ according to Corollary~\ref{cor:convex-combinations-of-isometries}. Since $\calX_Z$ consists of super-reflexive spaces (as an immediate consequence of the Assumptions~\ref{ass:framework}), the assertion follows.
\end{proof}

In case that $Z$ is not only super-reflexive, but satisfies an additional regularity property, it is natural to (try to) construct dilations on spaces that enjoy the same regularity property. This can be done by using the concept of super-properties which we recall next.

\begin{definition}
	Consider a property $(P)$ of Banach spaces which is invariant under isometric isomorphisms. We say that a Banach space $Z$ has \emph{super-$(P)$} if every Banach space $X$ finitely representable in $Z$ has $(P)$. If $(P)$ and super-$(P)$ are the same property, then we call $(P)$ a \emph{super-property}. 
\end{definition}

The question whether a Banach space $Z$ has super-$(P)$ is closely related to the question whether all ultra powers of $Z$ have $(P)$. For more information on super-properties we refer the interested reader to~\cite[Chapter~11]{Pis16} and~\cite[Chapter~8]{DJT95}.

\begin{theorem} \label{thm:super-property}
	Let $(P)$ be a super-property and let $Z$ be a super-reflexive Banach space such that $\ell^2(Z)$ has $(P)$. Further, let $\calT \subseteq \calL(X)$ by the weakly closed convex hull of all invertible isometries in $\calL(Z)$. Then $\calT$ has a simultaneous dilation in the class of all super-reflexive Banach spaces having property $(P)$.
\end{theorem}
\begin{proof}
	Consider the class $\calX_Z$ defined in Proposition~\ref{prop:X_Z_satisfies_assumptions}. The proposition implies that $\calX_Z$ contains $Z$ and satisfies the Assumptions~\ref{ass:framework}. Hence, according to Corollary~\ref{cor:convex-combinations-of-isometries} the set $\calT$ has simultaneous dilation in $\calX_Z$. Yet, since $(P)$ is a super-property, every element of $\calX_Z$ has $(P)$. As every element of $\calX_Z$ is also super-reflexive, this proves the assertion.
\end{proof}

This result applies to a rich collection of important super-properties. Among them, we explicitly mention uniform convexity, the UMD-property (see~\cite{HytNeeVer16}) and having prescribed type and cotype (see~\cite{DJT95}). As an example, we state the following dilation result for UMD spaces concretely.

\begin{corollary} \label{cor:umd-spaces}
	Let $Z$ be a UMD Banach space and let $\calT$ denote the weakly closed convex hull of all invertible isometries in $\calL(Z)$. Then $\calT$ has a simultaneous dilation in the class of all UMD Banach spaces.
\end{corollary}

\section{Application to \texorpdfstring{$L^p$-spaces}{Lebesgue spaces} and to Hilbert spaces} \label{sec:L-p}

\subsection{Dilations on \texorpdfstring{$L^p$-spaces}{Lebesgue spaces}} \label{subsec:L-p}

In this subsection we discuss how our toolkit gives the dilation theorem of Akcoglu--Sucheston on $L^p$-spaces. In fact, we obtain even a bit more, namely a simultaneous dilation of all positive contractions.

\begin{theorem} \label{thm:akcoglu-simultaneously}
	Let $p \in (1,\infty)$, let $(\Omega,\mu)$ be an arbitrary measure space and let $\calT$ denote the set of all positive linear contractions on $L^p(\Omega,\mu)$. Then $\calT$ has a simultaneous dilation in the class of all $L^p$-spaces. Moreover, the mappings $J$ and $Q$ from Definition~\ref{def:dilations}(b) can be chosen positive and the isometries $U_T$ from Definition~\ref{def:dilations}(b) can be chosen to be lattice isomorphisms.
\end{theorem}

The proof of Theorem~\ref{thm:akcoglu-simultaneously} relies on a non-canonical reduction procedure: first, we prove the theorem for $\Omega = [0,1]$, endowed with the Lebesgue measure; then we prove it for $\Omega = \{1,\dots,n\}$, endowed with the counting measure; and finally, we prove it for arbitrary measure spaces.

\begin{lemma} \label{lem:akcoglu-lp-0-1}
	Theorem~\ref{thm:akcoglu-simultaneously} is true if $\Omega = [0,1]$ and if $\mu$ is the Lebesgue measure.
\end{lemma}
\begin{proof}
	First note that every positive invertible isometry on $L^p(\Omega,\mu)$ is in fact a lattice isomorphism (this is actually true on every Banach lattice, see \cite{Abramovich1988} or \cite[Theorem~2.2.16]{Emelyanov2007}). According to~\cite[Theorem~2]{Grz90}, the set of positive invertible isometries on $L^p([0,1],\mu)$ is dense with respect to the weak operator topology in the set of all positive contractions on the same space. The assertion thus follows from Corollary~\ref{cor:convex-combinations-of-isometries} and from Remark~\ref{rem:positive-morphisms}(a).
\end{proof}

At first glance, it seems that dilations of convex combinations -- which constitute the most significant part of the present work -- do not play a role in the proof of Lemma~\ref{lem:akcoglu-lp-0-1} since the set of positive invertible isometries itself (and not only its convex hull) is weakly dense in the set of all positive contractions. However, the situation is not quite that simple: the weak operator closure of a set of invertible isometries might not have a simultaneous dilation in general (see the discussion after Proposition~\ref{prop:stop-limits}), but the strong operator closure has a simultaneous dilation according to Proposition~\ref{prop:stop-limits}. Thus, we need the convex hull to pass from the weak operator closure to the strong operator closure. 

\begin{remark} \label{rem:peller-power-approximation-vs-our-approach}
	One can even prove more than the density result \cite[Theorem~2]{Grz90} used above. In fact, Peller observed in~\cite[Section~4, Theorem~4 and Remark~3]{Pel81} that, for every regular operator $T$ on $L^p([0,1])$ with regular norm at most one, there exists a sequence of invertible isometries $(T_k)_{k \in \IN}$ on $L^p([0,1])$ such that all powers $T_k^n$ converge weakly to $T^n$ ($n \in \bbN_0$). This implies the Akcoglu--Sucheston dilation theorem on $L^p([0,1])$ (see the discussion after Proposition~\ref{prop:stop-limits}), and from this result one can deduce the theorem on general $L^p$-spaces (by the techniques used below). We find it important to compare this argument with our approach in more detail:
	
	As pointed out in the introduction, the main feature of our approach is that it splits dilation theorems into a purely dilation theoretic part which works on general Banach spaces (and which we have worked out in this paper) and into an approximation theoretic part. If we follows Peller's approach instead, proving dilation theorems seems to come down to an entirely approximation theoretical task. The price for this is that one has to show approximation results in a stronger topology than the weak operator topology and that one is not allowed to use convex combinations for the approximation.
	
	Such a stronger approximation giving the weak convergence of all powers works on $L^p([0,1])$ and -- to a certain extend -- on a class of rearrangement invariant Banach function spaces (see \cite[Section~6, Theorem~8]{Pel81}). However, the need to prove such stronger approximation results might turn out to be an obstacle if one intends to find dilation theorems on other classes of Banach spaces (compare also Open Problem~\ref{open-problem:lp-lq}). Besides, we point out that this approach does not yield simultaneous dilations.
\end{remark}

\begin{lemma} \label{lem:akcoglu-lp-n}
	Theorem~\ref{thm:akcoglu-simultaneously} is true if $\Omega = \{1,\dots, n\}$ for $n \in \bbN$ and if $\mu$ is the counting measure.
\end{lemma}
\begin{proof}
	Fix $n \in \bbN$. We write $\ell^p_n \coloneqq L^p(\Omega,\mu)$ and we use the abbreviation $L^p([0,1])$ for the $L^p$-space over $[0,1]$ endowed with the Lebesgue measure.
	
	There exist an $n$-dimensional vector sublattice $F \subseteq L^p([0,1])$, a positive contractive projection $P \in \calL(L^p([0,1]))$ with range $F$ and an isometric lattice homomorphism $J\colon \ell^p_n \to L^p([0,1])$ with range $F$. We define $Q \coloneqq J^{-1}P: L^p([0,1]) \to \ell^p_n$ and we set $S_T \coloneqq J T Q \in \calL(L^p([0,1]))$ for each $T \in \calT$.  Then $Q$ and $J$ are positive, each operator $S_T$ is a positive contraction on $L^p([0,1])$, and the diagramm
	\begin{center}
		\begin{tikzcd}
			L^p([0,1]) \arrow{rrr}{S_{T_1} \cdots S_{T_k}} & & & L^p([0,1]) \arrow{d}{Q} \\
			\ell^p_n \arrow{u}{J} \arrow{rrr}{T_1 \cdots T_k} & & & \ell^p_n
		\end{tikzcd}
	\end{center}
	commutes for all $k \in \bbN_0$ and $T_1,\dots,T_k \in \calT$. Thus, the assertion follows from Lemma~\ref{lem:akcoglu-lp-0-1}.
\end{proof}

\begin{remark}
	The reduction of Lemma~\ref{lem:akcoglu-lp-n} to Lemma~\ref{lem:akcoglu-lp-0-1} we used in the above proof is a bit curious: recall that Lemma~\ref{lem:akcoglu-lp-0-1} mainly relies on the fact that the positive invertible isometries on $L^p([0,1])$ are weakly dense in the set of positive contractions. On the other hand, on the finite dimensional spaces $\ell^p_n$ and for $p\not=2$ this is not even true for the subconvex hull of all positive (invertible) isometries. Indeed, every isometric positive matrix on such a space is a permutation matrix (this follows for instance from the fact that a linear isometry between two $L^p$-spaces is always disjointness preserving in case that $p \neq 2$, see \cite{Lamperti1958}). 
	Hence, every operator $T$ in the subconvex hull of those matrices maps the vector $e = (1, \ldots, 1)$ to a vector smaller than $e$, and so does the adjoint of $T$. There are, however, positive contractions on $\ell^p_n$ which do not behave this way. From a different view point, the set of all positive contractions on $\ell_n^p$ has a rich collection of extreme points. A complete characterization of these extreme points can be found in~\cite[Theorem~3]{Grz85}. 

	Hence, in order to apply Corollary~\ref{cor:convex-combinations-of-isometries} to $\ell^p_n$-spaces, we have to take a detour via the diffuse space $L^p([0,1])$; this is indeed quite surprising since in the proof of Theorem~\ref{thm:akcoglu-simultaneously} we prove the Akcoglu--Sucheston dilation theorem for general $L^p$-spaces by reducing it to the case of $\ell^p_n$-spaces - which is, in a sense, converse to the reduction of Lemma~\ref{lem:akcoglu-lp-n} to Lemma~\ref{lem:akcoglu-lp-0-1}. 
\end{remark}

The technique that we now use in the proof of Theorem~\ref{thm:akcoglu-simultaneously} is nowadays standard and goes back to Peller and W.B.~Johnson. It was, for instance, used by Akcoglu and Sucheston in~\cite[Section~4]{AkcSuc77}. Since we deal with simultaneous dilations here instead of dilations of a single operator, we think it is worthwhile to include the details.

\begin{proof}[Proof of Theorem~\ref{thm:akcoglu-simultaneously}]
	Throughout the proof, we use the abbreviation $L^p \coloneqq L^p(\Omega,\mu)$ and we let $\ell^p_n$ denote the space $\bbR^n$ (or $\bbC^n$) endowed with the $p$-norm.
	
	We call a finite collection of pairwise disjoint measurable subsets of $\Omega$ of strictly positive and finite measure a semi-partition of $\Omega$. The set $\mathcal{P}$ of all semi-partitions of $\Omega$ is a directed set with respect to refinement; we write $\alpha \ge \beta$ if $\alpha$ is finer than $\beta$. For each semi-partition $\alpha \in \mathcal{P}$ we can define its conditional expectation $\IE_{\alpha}$. Then the net $(\IE_{\alpha})_{\alpha \in \mathcal{P}}$ converges to the identity with respect to the strong operator topology.
	
	Fix $\alpha \in \mathcal{P}$. Since the range of $\IE_\alpha$ is a finite-dimensional vector sublattice of $L^p$, there exists an integer $n_\alpha \in \bbN$ and an isometric lattice homomorphism $J_\alpha \colon \ell^p_{n_\alpha} \to L^p$ whose range coincides with the range of $\IE_\alpha$. We set $Q_\alpha \coloneqq J_\alpha^{-1} \IE_\alpha\colon L^p \to \ell^p_{n_\alpha}$ and we define $S_{T,\alpha} \coloneqq Q_\alpha TJ_\alpha$ for each $T \in \calT_\alpha$. Then the diagram
	\begin{center}
		\begin{tikzcd}
			\ell^p_{n_\alpha} \arrow{rrrr}{S_{T_1,\alpha} \cdots S_{T_k,\alpha}} & & & & \ell^p_{n_\alpha} \arrow{d}{J_\alpha} \\
			L^p \arrow{u}{Q_\alpha} \arrow{rrrr}{ \IE_\alpha \, \cdot \, (\IE_\alpha T_1 \IE_\alpha) \, \cdots \, (\IE_\alpha T_k \IE_\alpha)} & & & & L^p
		\end{tikzcd}
	\end{center}
	commutes for each $k \in \bbN_0$ and all $T_1,\dots,T_k \in \calT$. Note that we need the additional operator $\IE_\alpha$ on  the very left of the lower horizontal arrow to ensure that the diagram also commutes in case that $k = 0$; indeed, the lower arrow equals $\IE_\alpha$ in this case (instead of $\Id_{L^p}$ which would be false).
	
	According to Lemma~\ref{lem:akcoglu-lp-n}, we can find an $L^p$-space $X_\alpha$, positive contractions $\tilde J_\alpha \colon \ell^p_{n_\alpha} \to  X_\alpha$ and $\tilde Q_\alpha \colon X_\alpha \to \ell^p_{n_\alpha}$ and isometric lattice isomorphisms $U_{\alpha, S} \in \calL(X_\alpha)$ (for each positive contraction $S$ on $\ell^p_{n_\alpha}$) such that
	\begin{align*}
		S_{1} \cdots S_{k} = \tilde Q_\alpha U_{\alpha, S_1} \cdots U_{\alpha, S_k} \tilde J_\alpha
	\end{align*}
	for all $k \in \bbN_0$ and all positive contractions $S_{1} ,\dots, S_{k}$ on $\ell^p_{n_\alpha}$. Choose an ultrafilter $\calU$ on $\mathcal{P}$ containing the filter base $\big\{ \{\alpha \in \mathcal{P}: \, \alpha \ge \beta\}: \; \beta \in \mathcal{P} \big\}$. Then the diagram
	\begin{center}
		\begin{tikzcd}
			\prod_\calU X_\alpha \arrow{rrrrr}{ (\prod_\calU U_{\alpha, S_{T_1,\alpha}}) \, \cdots \, (\prod_\calU U_{\alpha, S_{T_k,\alpha}}) } & & & & & \prod_\calU X_\alpha \arrow{d}{\prod_\calU \tilde Q_\alpha} \\
			\prod_\calU \arrow{u}{\prod_\calU \tilde J_\alpha} \ell^p_{n_\alpha} \arrow{rrrrr}{(S_{T_1,\alpha})^\calU \cdots (S_{T_k,\alpha})^\calU} & & & & & \prod_\calU \ell^p_{n_\alpha} \arrow{d}{\prod_\calU J_\alpha} \\
			(L^p)^\calU \arrow{u}{\prod_\calU Q_\alpha} \arrow{rrrrr}{\prod_\calU \big(\IE_\alpha \, \cdot \, (\IE_\alpha T_1 \IE_\alpha) \, \cdots \, (\IE_\alpha T_k \IE_\alpha) \big)} & & & & & (L^p)^\calU \arrow{d}{} \\
			L^p \arrow{u}{} \arrow{rrrrr}{T_1 \cdots T_k} & & & & & L^p
		\end{tikzcd}
	\end{center}
	commutes for each $k \in \bbN_0$ and all $T_1,\dots,T_k \in \calT$. Here, the mapping $L^p \to (L^p)^\calU$ between the first and the second line (counted from below) is the canonical injection and $(L^p)^\calU \to L^p$ between the second and the first line is the mapping induced by the weak limit along $\calU$ (which exists since $L^p$ is reflexive). We note that the diagram commutes between the first and the second line since the operator net $\big(\IE_\alpha \, \cdot \, (\IE_\alpha T_1 \IE_\alpha) \, \cdots \, (\IE_\alpha T_k \IE_\alpha)\big)_{\alpha \in \mathcal{P}}$ converges strongly to $T_1 \cdots T_k$ and since the ultrafilter $\calU$ is adapted to the order on $\mathcal{P}$. The diagram shows that $\calT$ has a simultaneous dilation with the required properties.
\end{proof}

It would be interesting to have a similar result as in Theorem~\ref{thm:akcoglu-simultaneously} -- or, say, at least Lemma~\ref{lem:akcoglu-lp-0-1} -- available on $L^p(L^q)$-spaces, too. The class of all $L^p(L^q)$-spaces itself is not ultra-stable, but the class of all bands in $L^p(L^q)$-spaces is ultra-stable (this follows from \cite[Corollary~8.8]{HLR91}) and thus fulfills Assumptions~\ref{ass:framework}. Hence, in order to apply our main result and its corollaries, it would be desirable to understand the weakly closed convex hull of all positive invertible isometries on such spaces.

\begin{open_problem} \label{open-problem:lp-lq}
	Let $p,q \in (1,\infty)$ and let $\calT$ denote the weak operator closure of the convex hull of all positive invertible isometries on $L^p([0,1]; L^q([0,1]))$. Does $\calT$ coincide with the set of all positive contractions? If not, can a good characterization of the elements of $\calT$ be given?
\end{open_problem}

\subsection*{Dilations on Hilbert spaces} \label{subsec:hilbert}

In the previous subsection we considered a dilation result for positive operators in the $L^p$-setting. On Hilbert spaces, on the other hand, one gets results for arbitrary contractions. For single operators, this is the well-known dilation theorem of Sz.-Nagy. We note that a standard proof of this result even yields a simultaneous dilation of all contractions on a given Hilbert space, as for example pointed out in \cite[Section~1.5.8]{Nik02}. Although this construction is not particularly difficult, we find it worthwhile to show that the same result can be obtained as a consequence of Corollary~\ref{cor:convex-combinations-of-isometries}. This emphasizes the universality of our approach.

\begin{theorem} \label{thm:hilbert-space-dilations-simultaneous}
	Let $H$ be a Hilbert space and $\calT \subseteq \calL(H)$ the set of all contractions on $H$. Then $\calT$ has a simultaneous dilation in the class of all Hilbert spaces.
\end{theorem}
\begin{proof}
	By a similar reduction argument as used in the proof of Theorem~\ref{thm:akcoglu-simultaneously} it suffices to establish the result if $H$ is finite dimensional. In this case, however, the convex hull of all (invertible) isometries in $\calL(H)$ coincides with the set of all contractions in $\calL(H)$; this is an easy consequence of the polar decomposition theorem for matrices. Hence, the assertion follows from Corollary~\ref{cor:convex-combinations-of-isometries}.
\end{proof}

\section{Outlook}\label{sec:outlook}

Our techniques do not work without adjustments to obtain non-trivial results on $L^1$-spaces: we require all our Banach spaces to be reflexive. However, as pointed out in Construction~\ref{constr:simple-dilation-ell-1}, it is not particularly difficult to find a dilation on a ``large'' $L^1$-space. We leave it to future research to find out whether our techniques can be adapted to $L^1$-spaces. Moreover, no attempt has been made to apply our results to non-commutative $L^p$-spaces; we also leave this as a task for the future. 

In view of our definition of a \emph{simultaneous dilation} (Definition~\ref{def:dilations}) it is worthwhile pointing out that there is a distinct interest in \emph{commutative} simultaneous dilations in the literature, especially in the Hilbert space case; see for instance \cite{Ando1963, Gacspar1969, Popescu1986, Stochel2001, Opela2006, SauPreprint} as well as \cite[Chapter~I]{SFBK10} and \cite[Section~4]{Levy2014} for this and related topics. Our approach does not yield commutative simultaneous dilations of commuting operators; we do not know whether commutative dilation theorems can be derived from our simultaneous dilations results.

A related question concerns the task of dilating a $C_0$-semigroup of operators instead of a single operator only; this question has been studied by Fendler for $L^p$-spaces~\cite{Fendler1997}  and by Konrad for $L^1$-spaces~\cite{Konrad2015}. Once a dilation theorem for single operators on a class of uniformly convex Banach spaces is established (as we do in the present work), one can mimic Fendler's general argument to obtain semigroup dilations.

\appendix

\section{Some observations from group theory} \label{appendix:group-theory}

In this appendix we explicitly write down a few simple observations from the theory of groups which are needed in the proofs of Theorems~\ref{thm:convex-combination} and~\ref{thm:convex-combinations-simultaneous}.

\begin{proposition} \label{prop:count-group-elements}
	Let $G$ be a finite abelian group, let $\varphi: G \times G \to G$, $(g_1,g_2) \mapsto g_1 g_2$. Then $\varphi$ is surjective and the preimage of each element $g \in G$ contains exactly $|G|$ elements.
\end{proposition}
\begin{proof}
	Obviously $\varphi$ is surjective. Since $G$ is abelian, $\varphi$ is a group homomorphism, and its kernel clearly consists of $|G|$ elements. Now, let $g \in G$. Since $g$ is contained in the range of $\varphi$ we have $|\varphi^{-1}(\{g\})| = |\ker \varphi| = |G|$.
\end{proof}

\begin{proposition} \label{prop:group-operation}
	Let $G$ be a finite group which operates on a finite set $X$. Fix $x \in X$ and denote the orbit of $x$ under $G$ by $G(x)$. Then $|G(x)|$ divides $|G|$. Moreover, if we define
	\begin{align*}
		G_y := \{g \in G: g(x) = y\}
	\end{align*}
	for each $y \in G(x)$, then the family $(G_y)_{y \in G(x)}$ is a partition of $G$ into $|G(x)|$ disjoint subsets and each set $G_y$ has the cardinality $\frac{|G|}{|G(x)|}$.
\end{proposition}
\begin{proof}
	Clearly, the sets $G_y$ (for $y \in G(x)$) are disjoint and form a partition of $G$ into $|G(x)|$ subsets, so it remains to show that all sets $G_y$ have the same cardinality. Let $y \in G(x)$ and fix an element $g_0 \in G_y$. Then the mapping
	\begin{align*}
		G_x & \to G_y \\
		g & \mapsto g_0 \, g
	\end{align*}
	is a bijection between $G_x$ and $G_y$. Hence, all sets $G_y$ have the same cardinality.
\end{proof}

\bibliographystyle{alpha}
\bibliography{dilations_scribble_biber}

\end{document}